\documentclass[a4paper,english,12pt]{article}

\usepackage{geometry}
\usepackage{pdfpages}
\usepackage[T1]{fontenc}
\usepackage[utf8]{inputenc}
\usepackage{babel}
\usepackage{amsmath}
\usepackage{amssymb}
\usepackage{amsthm}
\usepackage{textcomp}
\usepackage{enumitem}
\usepackage{dsfont}
\usepackage{mathrsfs}
\usepackage{cite}
\usepackage{color}
\usepackage{breqn}
\usepackage[colorlinks=true,urlcolor=blue,linkcolor=blue]{hyperref}

\newcommand{\R}{\mathbb{R}}

\newcommand{\Z}{\mathbb{Z}}
\newcommand{\N}{\mathbb{N}}
\newcommand{\C}{\mathbb{C}}

\newcommand{\T}{\mathbb{T}}

\newcommand{\classeC}{\mathcal{C}}

\renewcommand{\Re}{\textnormal{Re}}
\newcommand{\un}{\mathds{1}}

\newcommand{\half}{\frac{1}{2}}

\newcommand{\hilbert}{H}
\newcommand{\hamilton}{\mathcal{H}}

\newcommand{\intt}{\frac{1}{2\pi}\int_0^{2\pi}}

\usepackage{stmaryrd}
\newcommand{\li}{\llbracket}
\newcommand{\ri}{\rrbracket}

\renewcommand\d{\,{\mathrm d}}
\newcommand\e{\,{\mathrm e}}

\newcommand{\longrightarroww}[2] {\mathop{\longrightarrow}\limits_{#1}^{#2}}

\newtheorem{mydef}{Definition}[section]
\newtheorem{thm}[mydef]{Theorem}
\newtheorem{lem}[mydef]{Lemma}
\newtheorem{prop}[mydef]{Proposition}
\newtheorem{cor}[mydef]{Corollary}
\newtheorem{rk}[mydef]{Remark}

\title{The third order Benjamin-Ono equation on the torus : well-posedness, traveling waves and stability.}
\author{Louise Gassot}
\date{}
\newcommand{\Addresses}{{
  \footnotesize
\noindent
  \textsc{Département de mathématiques et applications, École normale supérieure, CNRS, PSL University, 75005 Paris, France\\Université Paris-Saclay, CNRS, Laboratoire de mathématiques d’Orsay, 91405, Orsay, France}\par\nopagebreak
  \noindent
  \textit{E-mail address :} \texttt{louise.gassot@math.u-psud.fr}
}}
\begin{document}
\maketitle
\abstract{
We consider the third order Benjamin-Ono equation on the torus
\[
\partial_tu
	= \partial_x \left( -\partial_{xx}u-\frac{3}{2}u H\partial_x u - \frac{3}{2}H(u\partial_x u) + u^3 \right).
\]
We prove that for any $t\in\R$, the flow map continuously extends to $H^s_{r,0}(\T)$ if $s\geq 0$, but does not admit a continuous extension to $H^{-s}_{r,0}(\T)$ if $0<s<\half$. Moreover, we show that the extension is not weakly sequentially continuous in $L^2_{r,0}(\T)$. We then classify the traveling wave solutions for the third order Benjamin-Ono equation in $L^2_{r,0}(\T)$ and study their orbital stability.
}

\tableofcontents

\section{Introduction}

We are interested in the third equation of the integrable Benjamin-Ono hierarchy on the torus
\begin{equation}\label{eq:bo4}
\partial_tu=\partial_x\left(-\partial_{xx}u-\frac{3}{2}u\hilbert\partial_x u-\frac{3}{2}\hilbert(u\partial_x u)+u^3\right).
\end{equation}
The operator $H$ is the Hilbert transform, defined as
\[
\hilbert f(x)=\sum_{n\in\Z\setminus{0}}-i\,\mathrm{sgn}(n)\widehat{f}(n)\e^{inx},
	\quad f=\sum_{n\in\Z}\widehat{f}(n)\e^{inx},
	\quad \widehat{f}(n)=\intt f(x)\e^{-inx}\d x.
\]

\subsection{Benjamin-Ono equations and integrability}

The Benjamin-Ono equation on the torus
\[
\partial_t u=\hilbert \partial_{xx}u-\partial_x(u^2),
\]
was introduced by Benjamin \cite{Benjamin1967} and Ono \cite{Ono1977} in order to describe long internal waves in a two-layer fluid of great depth.
This equation admits an infinite number of conserved quantities $\hamilton_k,k\geq 1$ (see Nakamura \cite{Nakamura1979} for a proof on the real line). The evolution equations associated to the conservation laws
\begin{equation}\label{eq:BOhierarchy}
\partial_t u=\partial_x(\nabla \hamilton_k(u))
\end{equation}
are the equations for the Benjamin-Ono hierarchy \cite{Matsuno1984}.

From Nakamura \cite{Nakamura1979-2} and Bock, Kruskal \cite{BockKruskal1979}, we know that the Benjamin-Ono equation admits a Lax pair
\[
\frac{\d}{\d t}L_u=[B_u,L_u],
\]
\[
L_u=Dh-T_u,	\quad B_u=iD^2+2iT_{D(\Pi u)}-2iDT_u.
\]
Here, $ D=-i\partial_x$ and $T_u$ is the Toeplitz operator on the Hardy space
\[
L^2_+(\T)=\{h\in L^2(\T)\mid \forall n<0,\quad \widehat{h}(n)=0\}
\]
defined as
\[
T_u:h\in L^2_+(\T)\mapsto \Pi(uh)\in L^2_+(\T),
\]
and $\Pi:L^2(\T)\to L^2_+(\T)$ is the Szeg\H{o} projector.
The Hamiltonians $\hamilton_k(u)$ are defined from the Lax operator $L_u$ as
\begin{equation}\label{eq:Hk_Lax}
\hamilton_k(u)=\langle L_u^{k}\un|\un\rangle.
\end{equation}
In particular, the Hamiltonian for equation \eqref{eq:bo4} is
\begin{equation}\label{eq:hamilton4}
\hamilton_4(u)+\half \hamilton_2(u)^2=\intt \left( \half (\partial_x u)^2-\frac{3}{4}u^2\hilbert\partial_xu+\frac{1}{4}u^4\right)\d x.
\end{equation}

In \cite{GerardKappeler2019}, Gérard and Kappeler constructed global Birkhoff coordinates for the Benjamin-Ono equation on the torus. In these coordinates, the evolution equations for the Benjamin-Ono hierarchy are easier to understand. Indeed, denote by $\Phi$ the Birkhoff map
\[
\Phi:u\in L^2_{r,0}(\T)\mapsto (\zeta_n(u))_{n\geq1}\in h^{\half}_+,
\]
where $L^2_{r,0}(\T)$ is the subspace of real valued functions in $L^2(\T)$ with zero mean, and 
\[
h^{\half}_+
	=\Big\{(\zeta_n)_{n\geq 1}\mid \sum_{n\geq 1}n|\zeta_n|^2<+\infty\Big\}.
\]
Then in the Birkhoff coordinates, equation \eqref{eq:BOhierarchy} of the hierarchy associated to $\hamilton_k$ 
becomes
\[
\partial_t \zeta_n=i\omega^{(k)}_n\zeta_n,
	\quad n\geq 1
\]
when the frequencies
\[
\omega^{(k)}_n
	=\frac{\partial (\hamilton_k\circ\Phi^{-1})}{\partial |\zeta_n|^2}
\]
are well-defined. For instance, this formula is valid if the sequence $\zeta(0)=(\zeta_n(0))_{n\geq 1}$ only has a finite number of nonzero terms, or in other words, if $\Phi^{-1}(\zeta(0))$ is a finite gap potential.
In this case, the frequencies $\omega^{(k)}_n$ only depend on the actions $|\zeta_p|^2$, and the evolution simply reads
\[
\zeta_n(t)=\zeta_n(0)\e^{i\omega^{(k)}_n(\zeta(0))t},
	\quad t\in\R,
	\quad n\geq 1.
\]
For the third equation of the hierarchy \eqref{eq:bo4}, the frequencies write
\begin{equation}\label{eq:omega_n^4}
\omega_n^{(4)}(\zeta)
	=n^3+n\sum_{p\geq 1}p|\zeta_p|^2-3\sum_{p\geq1}\min(p,n)^2|\zeta_p|^2+3\sum_{p,q\geq 1}\min(p,q,n)|\zeta_p|^2|\zeta_q|^2.
\end{equation}
More details about the frequencies $\omega_n^{(k)}$ and formula \eqref{eq:omega_n^4} can be found in Appendix~\ref{part:hierarchy}.

We refer to Saut \cite{Saut2018} for a detailed survey of the Benjamin-Ono equation and of its hierarchy.

\subsection{Main results}

Our first main result is the determination the well-posedness threshold for the third order Benjamin-Ono equation. For $s\in\R$, we use the notation 
\[
H^{-s}_{r,0}(\T)=\{u\in H^s(\T,\R)\mid \langle u|\un\rangle =0\}.
\]
We prove that the flow map is globally $\classeC^0$-well-posed (in the sense of Definitions 1 and 2 from \cite{GerardKappelerTopalov2019}) in $H^{s}_{r,0}(\T)$ when $s\geq 0$, but is not globally $\classeC^0$-well-posed in $H^{-s}_{r,0}(\T)$ when $0<s<\half$.

\begin{thm}
For all $t\in\R$, the flow map for equation \eqref{eq:bo4} $\mathcal{S}^t:u_0\mapsto u(t)$, defined for finite gap potentials, admits a continuous extension to $H^s_{r,0}(\T)$ for all $s\geq0$, but does not admit a continuous extension to $H^{-s}_{r,0}(\T)$ for $0<s<\half$.
\end{thm}
\begin{rk}
Note that if $s\geq \half$, the maps $t\in\R\mapsto u(t)$ constructed in this way are solutions to equation \eqref{eq:bo4} in the distribution sense.
\end{rk}
We also investigate the question of the sequential weak continuity for the flow map.
\begin{thm}
For all $t\in\R$, the extension of flow map for equation \eqref{eq:bo4} $\mathcal{S}^t:u_0\mapsto u(t)$ is weakly sequentially continuous in $H^{s}_{r,0}(\T)$ for $s>0$, but is not weakly sequentially continuous in $L^2_{r,0}(\T)$.
\end{thm}


In \cite{GerardKappelerTopalov2019}, Gérard, Kappeler and Topalov proved that the flow map for the Benjamin-Ono equation is globally $\classeC^0$-well-posed in $H^{s}_{r,0}(\T)$ for $s>-\half$, whereas from \cite{PavaHakkaev2010} there is no continuous extension of the flow map to $H^s_{r,0}(\T)$ when $s<-\half$. We expect that the well-posedness threshold on the torus increases by $\half$ for each new equation in the hierarchy~: for the equation corresponding to the $k$-th Hamiltonian $\hamilton_k$, $k\geq 4$, the threshold should be $H^{\frac{k}{2}-2}_{r,0}(\T)$  (see Remarks \ref{rk:GWP} and~\ref{rk:WP}). Note that all the equations for the Benjamin-Ono hierarchy have critical Sobolev exponent $-\half$.

Let us mention former approaches to the Cauchy problem for higher order Benjamin-Ono equations.
Tanaka \cite{Tanaka2019} considered more general third order type Benjamin-Ono equations on the torus 
\[
\partial_t u=\partial_x(-\partial_{xx}u-c_1u\hilbert\partial_x u-c_2\hilbert(u\partial_x u)+u^3),
\]
and proved local well-posedness in $H^s(\T)$ for $s>\frac{5}{2}$. He deduced global well-posedness in $H^s(\T)$, $s\geq 3$ for the integrable case $c_1=c_2=\frac{3}{2}$.

On the real line, Feng and Han \cite{FengHan1996} proved local well-posedness in $H^s(\R)$, $s\geq 4$ for the third equation of the Benjamin-Ono hierarchy \eqref{eq:bo4}. Considering more general third order type Benjamin-Ono equations under the form
\[
\partial_t u-b\hilbert\partial_{xx}u-a\partial_{xxx}u=cv\partial_xv-d\partial_x(v\hilbert \partial_x v+\hilbert(v\partial_xv)),
\]
Linares, Pilod and Ponce \cite{LinaresPilodPonce2011} established local well-posedness in $H^s(\R)$, $s\geq2$, then Molinet and Pilod \cite{MolinetPilod2012} proved global well-posedness in $H^s(\R)$, $s\geq 1$.

Concerning Benjamin-Ono equations of fourth order on the torus and on the real line, Tanaka \cite{Tanaka2019-2} proved local well-posedness in $H^s$, $s>\frac{7}{2}$ for a more general family of fourth order type Benjamin-Ono equations, and deduced global well-posedness in $H^s$, $s\geq 4$ in the integrable case.
\newline

Our second main result is the classification of the traveling waves for the third order Benjamin-Ono equation in $L^2_{r,0}(\T)$, i.e.\@ the solutions to \eqref{eq:bo4} under the form $u(t,x)=u_0(x+ct)$, $t\in\R$, $x\in\T$, $u_0\in L^2_{r,0}(\T)$.

\begin{mydef}
For $N\geq 1$, we say that $u\in L^2_{r,0}(\T)$ is a $N$ gap potential if the set $\{n\geq 1\mid \zeta_n(u)\neq 0\}$, where $\Phi(u)=(\zeta_n(u))_{n\geq1}$, is finite and of cardinality $N$.
\end{mydef}
\begin{thm}
A potential $u_0\in L^2_{r,0}(\T)$ defines a traveling wave for equation \eqref{eq:bo4} if and only if
\begin{itemize}
\item either $u_0$ is a one gap potential ;
\item either $u_0$ is a two gap potential, and the two nonzero indexes $p<q$ satisfy, with $\gamma_p=|\zeta_p|^2$ and $\gamma_q=|\zeta_q|^2$,
\[
0< \gamma_p< \half\left(p+\sqrt{p^2+4q\frac{p+q}{3}}\right)
\]
and
\[
\gamma_q=\frac{q\frac{p+q}{3}-\gamma_p^2+p\gamma_p}{2\gamma_p+q}.
\]

\end{itemize}
\end{thm}

Note that from \cite{GerardKappeler2019}, the one gap potentials are the only traveling wave solutions to the Benjamin-Ono equation ; they have been characterized by Amick and Toland \cite{AmickToland1991}.
\newline

Our last main result answers the question of orbital stability for these two types of traveling waves.

\begin{mydef}
Let $u_0\in L^2_{r,0}(\T)$ be a one gap traveling wave. We say the $u_0$ is orbitally stable if for all $\varepsilon>0$, there exists $\delta>0$ such that if $v$ is a solution to \eqref{eq:bo4} with initial condition $v_0\in L^2_{r,0}(\T)$ such that $\|v_0-u_0\|_{L^2(\T)}\leq \delta$, then
\[
\sup_{t\in\R}\inf_{\theta\in\T}\|v(t)-u_0(\cdot+\theta)\|_{L^2(\T)}\leq \varepsilon.
\]
\end{mydef}

\begin{thm}
The one gap traveling waves are orbitally stable, whereas the two gap traveling waves are orbitally unstable.
\end{thm}

For the Benjamin-Ono equation, Pava and Natali \cite{PavaNatali2008} proved the orbital stability of the traveling wave solutions in $H^{\half}_{r,0}(\T)$. In \cite{GerardKappelerTopalov2019}, Gérard, Kappeler and Topalov improved the orbital stability of these solutions to $H^{-s}_{r,0}(\T)$, $0\leq s<\half$.

\paragraph{Plan of the paper} The paper is organized as follows. We first prove the well-posedness threshold for the third order Benjamin-Ono equation \eqref{eq:bo4} in Section \ref{part:GWP}. Finally, in Section \ref{part:traveling_waves}, we classify the traveling wave solutions and study their orbital stability properties.

In Appendix \ref{part:hierarchy}, we describe how to compute the Hamiltonians $\hamilton_k$ and frequencies $\omega^{(k)}_n=\frac{\partial \hamilton_k\circ\Phi^{-1}}{\partial |\zeta_n|^2}$ in terms of the action variables $|\zeta_p|^2$. In Appendix \ref{part:appendix}, we retrieve the Hamiltonian and frequencies of the third order Benjamin-Ono equation (see formulas~\eqref{eq:hamilton4} and \eqref{eq:omega_n^4}) by starting from the definition \eqref{eq:Hk_Lax} of the higher order Hamiltonians. In Appendix~\ref{part:appendix_structure}, we provide an alternative proof of a result from \cite{TzvetkovVisciglia2014} about the structure of the higher order Hamiltonians by using formula~\eqref{eq:Hk_Lax}, which may be of independent interest.

\paragraph{Acknowledgements} The author would like to thank her PhD advisor Professor P.\@ Gérard for introducing her to this problem, and for his continuous support and advices.

\section{Well-posedness threshold for the fourth Hamiltonian}\label{part:GWP}

Let $N\in\N$ and let $\mathcal{U}_N$ be the set
\[
\mathcal{U}_N
	=\{u\in L^2_{r,0}(\T)\mid \zeta_N(u)\neq 0 ,\quad \zeta_j(u)=0 \quad\forall j>N\}.
\]
We know from \cite{GerardKappeler2019}, Theorem 3, that the restriction of the Birkhoff map $\Phi$ to $\mathcal{U}_N$ is a real analytic diffeomorphism onto some Euclidean space. In Birkhoff coordinates, the evolution along the flow of equation \eqref{eq:bo4} for an initial data $u_0\in\mathcal{U}_N$ writes
\[\begin{cases}
\partial_t\zeta_n=i\omega_n^{(4)}(u_0)\zeta_n
\\
\zeta_n(0)=\zeta_n(u_0)
\end{cases},
	\quad n\geq 1,
\]
where for all $n\geq 1$, the frequencies $\omega_n^{(4)}(u_0)$ are given by \eqref{eq:omega_n^4}
\begin{align*}
\omega_n^{(4)}(u_0)
	=n^3+n\sum_{p\geq 1}p|\zeta_p(u_0)|^2-3\sum_{p\geq1}\min(p,n)^2|\zeta_p(u_0)|^2+3\sum_{p,q\geq 1}\min(p,q,n)|\zeta_p(u_0)|^2|\zeta_q(u_0)|^2.
\end{align*}
This implies that
\[
\zeta_n(u(t))=\zeta_n(u_0)\e^{i\omega_n^{(4)}(u_0)t},
	\quad t\in\R, \quad n\geq1.
\]
Therefore, for any finite gap inifial data $u_0$, belonging to some of the sets $\mathcal{U}_N$, the flow map $\mathcal{S}^t:u_0\in\mathcal{U}_N\mapsto u(t)\in\mathcal{U}_N$ is well-defined.

In part~\ref{subsection:GWP}, we prove that for all $t\in\R$, this flow map extends by continuity to $H^s_{r,0}(\T)$ for $s\geq 0$. We also show that the extension is sequentially weakly continuous in $H^s_{r,0}(\T)$ for $s> 0$, but not in $L^2_{r,0}(\T)$. In part \ref{subsection:ill_posed}, we prove that the flow map does not extend by continuity to $H^{-s}_{r,0}(\T)$ for $s>0$. This gives a threshold for the global $\classeC^0$-well-posedness of the third order Benjamin-Ono equation in the sense of Definitions 1 and 2 from \cite{GerardKappelerTopalov2019}.

\subsection{Well-posedness in \texorpdfstring{$H^s_{r,0}(\T)$, $s\geq 0$}{H^s, s>=0}}\label{subsection:GWP}

\begin{prop}\label{prop:GWP}
Let $s\geq 0$. For any $u_0\in H^s_{r,0}(\T)$, there exists a continuous map $t\in\R\mapsto \mathcal{S}^t(u_0)=u(t)\in H^s_{r,0}(\T)$ with $u(0)=u_0$ such that the following holds.

For any finite gap sequence $(u_0^k)_k$ converging to $u_0$ in $H^s_{r,0}(\T)$, for any $t\in\R$, $u_k(t)=\mathcal{S}^t(u_0^k)$ converges to $u(t)$ in $H^s_{r,0}(\T)$ as $k$ goes to infinity.

Moreover, the extension of the flow map $\mathcal{S}:u_0\in H^s_{r,0}(\T)\mapsto (t\mapsto u(t))\in \classeC(\R,H^s_{r,0}(\T))$ is continuous.
\end{prop}

Recall that 
from \cite{GerardKappelerTopalov2019}, as mentioned in the proofs of Proposition 2 and Theorem 8, we have the following result. For $s\geq 0$, the Birkhoff map $\Phi$ defines a homeomorphism between $H^s_{r,0}(\T)$ and the space
\[
h^{\half+s}_+
	=\Big\{(\zeta_n)_{n\geq 1}\mid \sum_{n\geq 1}n^{1+2s}|\zeta_n|^2<+\infty\Big\}.
\]
The proof of Proposition \ref{prop:GWP} therefore relies on the following sequential convergence result obtained after applying the Birkhoff map.

\begin{lem}\label{lem:convergence_zeta}
Fix $s\geq0$. Let $\zeta^k=(\zeta^k_n)_{n\geq 1}$, $k\in\N$, and $\zeta$ be elements of $h^{\half+s}_+$ such that
\(
\|\zeta^k-\zeta\|_{h^{\half+s}_+}\longrightarroww{k\to+\infty}{}0.
\)
Then for all $t\in\R$,
\[
\|(\zeta^k_n\e^{i\omega_n^{(4)}(\zeta^k)t})_n-(\zeta_n\e^{i\omega_n^{(4)}(\zeta)t})_n\|_{h^{\half+s}_+}\longrightarroww{k\to+\infty}{}0,
\]
where the convergence is uniform on bounded time intervals.
\end{lem}

\begin{proof}
Note that since $(\zeta^k)_k$ converges to $\zeta$ in $h^{\half+s}_+$, then for all $n\in\N$, formula \eqref{eq:omega_n^4} for $\omega_n^{(4)}(\zeta^k)$ implies that $\omega_n^{(4)}(\zeta^k)$ converges to $\omega_n^{(4)}(\zeta)$ as $k$ goes to infinity.

Let $\varepsilon>0$. Fix $K\in\N$ such that for all $k\geq K$,
\[
\|\zeta^k-\zeta\|_{h^{\half+s}_+}
	\leq \varepsilon.
\]
Using that $\zeta\in h^{\half+s}_+$, fix $N\in\N$ such that
\[
\Big(\sum_{n\geq N}n^{1+2s}|\zeta_n|^2\Big)^{\half}\leq\varepsilon.
\]

Now, if $k\geq K$,
\begin{align*}
\|(\zeta^k_n\e^{i\omega_n^{(4)}(\zeta^k)t})_n-(\zeta_n\e^{i\omega_n^{(4)}(\zeta)t})_n\|_{h^{\half+s}_+}
	&\leq \|(\zeta^k_n)_n-(\zeta_n)_n\|_{h^{\half+s}_+}
+\|(\zeta_n(\e^{i\omega_n^{(4)}(\zeta^k)t}-\e^{i\omega_n^{(4)}(\zeta)t}))_n\|_{h^{\half+s}_+}\\
	&\leq 3\varepsilon+\left(\sum_{n=0}^{N-1}n^{1+2s}|\zeta_n(\e^{i\omega_n^{(4)}(\zeta^k)t}-\e^{i\omega_n^{(4)}(\zeta)t})|^2\right)^{\half},
\end{align*}
which is less than $4\varepsilon$ for $k$ large enough by convergence term by term of the elements in the sum. Moreover, this convergence is uniform on bounded time intervals.
\end{proof}

\begin{proof}[Proof of Proposition \ref{prop:GWP}]
Let $s\geq 0$ and $u_0\in H^s_{r,0}(\T)$. Fix $t\in \R$, and a sequence of finite gap initial data $(u_0^k)_k$ converging to $u_0$ in $H^s_{r,0}(\T)$. 

We first establish that for all $t\in\R$, $(u_k(t))_k$ has a limit in $H^s_{r,0}(\T)$ as $k$ goes to $+\infty$. By assumption, $\Phi(u_0^k)$ converges to $\Phi(u_0)$ in $h^{\half+s}_+$.
Define the sequence $\zeta(t)$ by
\[
\zeta_n(t):=\zeta_n(u_0)\e^{i\omega_n^{(4)}(u_0)t}, \quad n\in\N.
\]
Lemma \ref{lem:convergence_zeta} immediately implies that the sequence $(\Phi(u_k(t)))_k$ converges to $\zeta(t)$ in $h^{\half+s}_+$. Since $\Phi^{-1}$ defines a continuous application from $h^{\half+s}_+$ to $H^s_{r,0}(\T)$, we deduce that $u_k(t)$ converges in $H^s_{r,0}(\T)$ to $u(t):=\Phi^{-1}(\zeta(t))$. Moreover, the convergence is uniform on bounded time intervals.

We now prove the continuity of the flow map $\mathcal{S}^t$. Let $u_0^k\in H^s_{r,0}(\T)$, $k\in\N$, be a sequence of initial data converging to some $u_0$ in $H^s_{r,0}(\T)$. Then $\Phi(u_0^k)$ converges to $\Phi(u_0)$ in $h^{\half+s}_+$, and the above Lemma \ref{lem:convergence_zeta} again implies that $\Phi(u_k(t))$ converges to $\Phi(u(t))$ in $h^{\half+s}_+$. In other terms, $u_k(t)$ converges to $u(t)$ in $H^s_{r,0}(\T)$, where again this convergence is uniform on bounded intervals.
\end{proof}

\begin{cor}
For all $s>0$ and all $t\in\R$, the extension of the flow map restricted to $H^s_{r,0}(\T)$~: $u_0\in H^s_{r,0}(\T)\mapsto u(t)\in H^s_{r,0}(\T)$ is sequentially weakly continuous.
\end{cor}


\begin{proof}
Let $u_0^k\in H^s_{r,0}(\T)$, $k\in\N$, be a sequence weakly converging in $H^s_{r,0}(\T)$ to $u_0\in H^s_{r,0}(\T)$. Since  the embedding $H^s_{r,0}(\T)\hookrightarrow L^2_{r,0}(\T)$ is compact, $(u_0^k)_k$ is strongly convergent to $u_0$ in $L^2_{r,0}(\T)$. By continuity of the flow map $\mathcal{S}^t$, one deduces that $(u_k(t))_k$ converges strongly to $u(t)$ in $L^2_{r,0}(\T)$.  This implies that $(u_k(t))_k$ converges weakly to $u(t)$ in $H^s_{r,0}(\T)$.
\end{proof}

\begin{prop}
For all $t\in\R^*$, the extension of the flow map restricted to $L^2_{r,0}(\T)$~: $u_0\in L^2_{r,0}(\T)\mapsto u(t)\in L^2_{r,0}(\T)$ is not sequentially weakly continuous.
\end{prop}

\begin{proof}
Fix $t\in\R^*$ and $u_0\in L^2_{r,0}(\T)\setminus\{0\}$. We construct a sequence $(u_0^k)_k$ in $L^2_{r,0}(\T)$ weakly convergent to $u_0$ in $L^2_{r,0}(\T)$ but such that $u_k(t)=\mathcal{S}^t(u_0^k)$ is not weakly convergent to $u(t)=\mathcal{S}^t(u_0)$ in $L^2_{r,0}(\T)$.

Let $\alpha>0$ to be chosen later. For $k\in\N$, we choose $(\zeta_p(u_0^k))_p$ converging weakly to $(\zeta_p(u_0))_p$ in $h^{\half}_+$ (so that $u_0^k$ converges weakly to $u_0$ in $L^2_{r,0}(\T)$) and such that
\[
|\zeta_p(u_0^k)|^2=|\zeta_p(u_0)|^2+\frac{\alpha}{p}\delta_{k,p},
	\quad p\geq 1.
\]
For instance, for $p\neq k$ we choose $\zeta_p(u_0^k)=\zeta_p(u_0)$, and for $p=k$, we choose $\zeta_k(u_0^k)=\sqrt{|\zeta_k(u_0)|^2+\frac{\alpha}{k}}\frac{\zeta_k(u_0)}{|\zeta_k(u_0)|}$ if $\zeta_k(u_0)\neq 0$ and $\zeta_k(u_0^k)=\frac{\alpha}{k}$ if $\zeta_k(u_0)= 0$.

Fix $t\neq 0$. If $u_k(t)$ was weakly convergent to $u_0$ in $L^2_{r,0}(\T)$, then $(\zeta_p(u_k(t))_p$ would converge weakly to $(\zeta_p(u(t))_p$ in $h^{\half}_+$, and therefore component by component :
\[
\zeta_p(u_0^k)\e^{i\omega^{(4)}_p(u_0^k)t}
	\longrightarroww{k\to+\infty}{} \zeta_p(u_0)\e^{i\omega^{(4)}_p(u_0)t},
	\quad p\geq 1.
\]
In particular, let $p\geq1$ such that $\zeta_p(u_0)\neq 0$. Then there exists a sequence $(n_k)_k$ of integers such that
\[
\omega^{(4)}_p(u_0^k)t+2\pi n_k
	\longrightarroww{k\to+\infty}{} \omega^{(4)}_p(u_0)t.
\]
From the expression \eqref{eq:omega_n^4} of $\omega^{(4)}_p(u_0^k)$ and the strong convergence of $(\zeta_p(u_0^k))_p$ to $(\zeta_p(u_0))_p$ in $\ell^2_+=\{(\zeta_p)_{p\geq 1}\mid \sum_{p\geq1}|\zeta_p|^2<+\infty\}$ by compactness, we get
\[
\sum_{p=1}^{+\infty}p|\zeta_p(u_0)|^2+\alpha+\frac{2\pi n_k}{t}
	=\sum_{p=1}^{+\infty}p|\zeta_p(u_0^k)|^2+\frac{2\pi n_k}{t}
	\longrightarroww{k\to+\infty}{} \sum_{p=1}^{+\infty}p|\zeta_p(u_0)|^2.
\] 
We get a contradiction by choosing $\alpha\not\in \frac{2\pi}{t}\Z$.
\end{proof}
\subsection{Ill-posedness in \texorpdfstring{$H^{-s}_{r,0}(\T)$, $s>0$}{H^{-s}, s>0}}\label{subsection:ill_posed}

\begin{prop}
For all $t>0$, there is no continuous local extension of the flow map $\mathcal{S}^t$ to $H^{-s}_{r,0}(\T)$ for $0<s<\half$ in the distribution sense.
\end{prop}


\begin{proof}
Let us fix $0<s<\half$ and an initial data $u_0\in H^{-s}_{r,0}(\T)\setminus L^2_{r,0}(\T)$. From \cite{GerardKappelerTopalov2019}, Theorem 5, the Birkhoff map extends by continuity as an homeomorphism
\[
\Phi:u\in H^{-s}_{r,0}\mapsto \Phi(u)=(\zeta_n(u))_{n\geq1}\in h^{\half-s}_+
\]
where
\[
h^{\half-s}_+
	=\Big\{(\zeta_n)_{n\geq 1}\mid \sum_{n\geq 1}n^{1-2s}|\zeta_n|^2<+\infty\Big\}.
\]
Therefore, $(\zeta_n(u_0))_{n\geq 1}:=\Phi(u_0)\in h^{\half-s}_+$ is well defined. Let $u_0^k$, $k\in\N$, be a sequence finite gap initial data, to be chosen later, such that $u_0^k$ converges in $H^{-s}_{r,0}(\T)$ to $u_0$. Write
\[
\Phi(u_0^k)=(\zeta_n(u_0^k)\un_{n\leq N_k})_n,
	\quad k\in\N.
\]
Since $u_0^k$ is a finite gap potential, it belongs to $L^2_{r,0}(\T)$. Recall that
\[
\omega_n^{(4)}(u_0^k)-n\sum_{p=1}^{N_k}p|\zeta_p(u_0^k)|^2
	=\widetilde{\omega_n}(u_0^k)
\]
where
\[
\widetilde{\omega_n}(u_0^k)
	=n^3-3\sum_{p=1}^{N_k}\min(p,n)^2|\zeta_p(u_0^k)|^2+3\sum_{p=1}^{N_k}\sum_{q=1}^{N_k}\min(p,q,n)|\zeta_p(u_0^k)|^2|\zeta_q(u_0^k)|^2.
\]
Since $u_0^k$ converges to $u_0$ in $H^{-s}_{r,0}(\T)$, the series $\sum_{p\geq 1}|\zeta_p(u_0)|^2$ is convergent, and
\begin{equation*}
\sum_{p\geq 1}|\zeta_p(u_0^k)|^2
	\longrightarroww{k\to+\infty}{} \sum_{p\geq 1}|\zeta_p(u_0)|^2.
\end{equation*}
In particular, the term $\widetilde{\omega_n}(u_0^k)$ converges as $k$ goes to infinity to
\[
\widetilde{\omega_n}(u_0)
	=n^3-3\sum_{p=1}^{+\infty}\min(p,n)^2|\zeta_p(u_0)|^2+3\sum_{p=1}^{+\infty}\sum_{q=1}^{+\infty}\min(p,q,n)|\zeta_p(u_0)|^2|\zeta_q(u_0)|^2.
\]

For $k\in\N$, let 
\[\tau_k:=\sum_{p=1}^{N_k}p|\zeta_p(u_0^k)|^2=\frac{1}{2}\|u_0^k\|_{L^2(\T)}^2
\]
and
\[
v_k(t,\cdot):=u_k(t,\cdot-\tau_k t).
\]
We use the following identity from the proof of Proposition B.1.\@ in \cite{GerardKappeler2019} :
\[
\zeta_n(u(\cdot+\tau))=\zeta_n(u)\e^{i\tau n},
	\quad \tau\in\R,\quad u\in L^2_{r,0}(\T),
\]
to deduce that for $n\in\N$,
\begin{align*}
\zeta_n(v_k(t))
	&=\zeta_n(u_k(t))\e^{-in\tau_k t}\\
	&=\zeta_n(u_0^k)\e^{i(\omega^{(4)}_n(u_0^k)-n\tau_k) t}.
\end{align*}
Since
\[
\omega_n^{(4)}(u_0^k)-n\tau_k
	\longrightarroww{k\to+\infty}{}\widetilde{\omega_n}(u_0),
\]
the sequence \((\zeta_n(v_k(t)))_k\)
is convergent :
\begin{equation}\label{eq:limitv1}
\zeta_n(v_k(t))
	\longrightarroww{k\to +\infty}{} \zeta_n(u_0)\e^{i\widetilde{\omega_n}(u_0)t}.
\end{equation}

Let $t>0$. If there was a local extension of the flow map $\mathcal{S}^t$ in the distribution sense, then $u_k(t)$ would be weakly convergent to $u(t)$ in $H^{-s}_{r,0}(\T)$. Applying the Birkhoff map, which is weakly sequentially continuous (see \cite{GerardKappelerTopalov2019}, Theorem 6), $\Phi(u_k(t))$ would converge weakly to $\Phi(u(t))$ in $h^{\half-s}_+$. In particular, for all $n$,
\begin{equation}\label{eq:limitv2}
\zeta_n(v_k(t))\e^{i\tau_k n t}
	=\zeta_n(u_k(t))
	\longrightarroww{k\to+\infty}{} \zeta_n(u(t)).
\end{equation}

We deduce from \eqref{eq:limitv1} and \eqref{eq:limitv2} that if $\zeta_n(u_0)\neq 0$, then
\begin{equation}\label{eq:limitv3}
\e^{i\tau_k n t}\longrightarroww{k\to+\infty}{} \frac{\zeta_n(u(t))}{\zeta_n(u_0)}\e^{-i\widetilde{\omega_n}(u_0)t}.
\end{equation}
We construct the sequence $(u_0^k)_k$ in order to contradict this latter point. Let $n\in \N$ such that $\zeta_n(u_0)\neq 0$. Fix $k\in \N$. From the fact that $u_0$ does not belong to $L^2_{r,0}(\T)$,
\[
\sum_{p> k}p|\zeta_p(u_0)|^2=+\infty,
\]
therefore one can choose $N_k\geq k+1$ such that
\[
\sum_{p=k+1}^{N_k}p|\zeta_p(u_0)|^2\geq \frac{2\pi}{nt}.
\]
Let $0<\alpha_k<1$ such that there exists an integer $m_k$ such that
\[
\sum_{p\leq k}p|\zeta_p(u_0)|^2+\alpha_k\sum_{p=k+1}^{N_k}p|\zeta_p(u_0)|^2
	=\frac{1}{nt}(k\pi+2\pi m_k).
\]

We define $u_0^k$ by
\[
\zeta_p(u_0^k)
	=
\begin{cases}
\zeta_p(u_0) \text{ if } p\leq k
\\
\sqrt{\alpha_k}\zeta_p(u_0) \text{ if } k<p\leq N_k
\\
0 \text{ if } N_k<p
\end{cases},
	\quad p\in\N.
\]
By construction, $u_0^k$ is finite gap and converges to $u_0$ in $H^{-s}_{r,0}(\T)$. However,
\begin{align*}
\tau_k
	&=\sum_{p\leq k}p|\zeta_p(u_0)|^2+\alpha_k\sum_{p=k+1}^{N_k}p|\zeta_p(u_0)|^2\\
	&=\frac{1}{nt}(k\pi+2\pi m_k),
\end{align*}
which implies that
\[
\e^{i\tau_k n t}=(-1)^k.
\]
In particular, the sequence $(\e^{i\tau_k n t})_k$ is not convergent, and we get a contradiction with \eqref{eq:limitv3}.
\end{proof}

\begin{rk}\label{rk:GWP}
We expect that with a similar argument, one can prove the following fact. For the higher equations of the hierarchy, the well-posedness threshold increases by $\half$ for each equation (see Remark \ref{rk:WP}).
\end{rk}

\section{Traveling waves for the fourth Hamiltonian}\label{part:traveling_waves}

In this part, we classify all traveling wave solutions to equation \eqref{eq:bo4}
\[
\partial_x(-cu-\partial_{xx}u-\frac{3}{2}u^2H\partial_x u-\frac{3}{2}H(u\partial_x u)+u^3)=0.
\]
A traveling wave of speed $c\in\R$ is a solution to \eqref{eq:bo4} of the form $u(t,x)=u_0(x+ct)$.

The argument of \cite{GerardKappeler2019}, Proposition B.1., applies for the higher equations of the Benjamin-Ono hierarchy, implying that a potential $u_0$ is a traveling wave solution to \eqref{eq:bo4} of speed $c\in\R$ if and only if for all $t\in\R$ and $n\geq 1$,
\[
\e^{icnt}\zeta_n(u_0)=\e^{i\omega_n^{(4)}(u_0)t}\zeta_n(u_0),
\]
or in other words :
\begin{equation}\label{condition}
\forall n\geq 1, \text{ if } \zeta_n(u_0)\neq 0, \text{ then } cn=\omega_n^{(4)}(u_0).
\end{equation}
In particular, all one gap potentials are traveling wave solutions. Note that these potentials are the only traveling wave solutions to the Benjamin-Ono equation and write (see \cite{GerardKappeler2019}, Appendix B)
\[
u_0(x)=\frac{pw\e^{ipx}}{1-w\e^{ipx}}+\frac{p\overline{w}\e^{-ipx}}{1-\overline{w}\e^{-ipx}}
\]
with the nonzero gap being at index $p\geq 1$ and $w=\frac{\zeta_p(u_0)}{\sqrt{p+\gamma_p(u_0)}}$. As we will see in this section, the one gap potentials are not the only traveling wave solutions for equation \eqref{eq:bo4}.

In part \ref{subsection:classification}, we first show that the traveling waves are necessarily one gap and two gap potentials, then provide a classification of the two gap traveling waves in term of their actions. In part \ref{subsection:stability}, we prove that the one gap traveling waves are orbitally stable whereas the two gap traveling waves are orbitally unstable.

\subsection{Classification of traveling wave solutions}\label{subsection:classification}

In the following, it will be more convenient to work with the actions $\gamma_p=|\zeta_p(u_0)|^2$. Formula \eqref{eq:omega_n^4} for $\omega_n^{(4)}(u_0)$
\begin{align*}
\omega_n^{(4)}(u_0)
	&=n^3+n\sum_{p\geq 1}p\gamma_p-3\sum_{p\geq1}\min(p,n)^2\gamma_p+3\sum_{p,q\geq 1}\min(p,q,n)\gamma_p\gamma_q
\end{align*}
shows that $\frac{\omega_n^{(4)}(u_0)}{n}$ is equivalent to $n^2$ as $n$ goes to infinity, therefore, from condition \eqref{condition}, the traveling waves for the third equation of the hierarchy \eqref{eq:bo4} are necessarily finite gap solutions.

\begin{prop}
Let $u_0$ be a two gap potential, and $p<q$ be the indices of the two nonzero gaps with gaps $\gamma_p>0$ and $\gamma_q>0$. Then $u_0$ is a traveling wave for equation \eqref{eq:bo4} if and only if
\[
0< \gamma_p< \half\left(p+\sqrt{p^2+4q\frac{p+q}{3}}\right)
\]
and
\[
\gamma_q=\frac{q\frac{p+q}{3}+p\gamma_p-\gamma_p^2}{2\gamma_p+q}.
\]
\end{prop}

\begin{proof}
Let $u_0$ be such a two gap potential. Then $u_0$ is a traveling wave if and only if $\frac{\omega_p^{(4)}(u_0)}{p}=\frac{\omega_q^{(4)}(u_0)}{q}$, where
\[
\frac{\omega_p^{(4)}(u_0)}{p}
	=p^2+(p\gamma_p+q\gamma_q)-3p(\gamma_p+\gamma_q)+3(\gamma_p^2+2\gamma_p\gamma_q+\gamma_q^2)
\]
and
\[
\frac{\omega_q^{(4)}(u_0)}{q}
	=q^2+(p\gamma_p+q\gamma_q)-3(\frac{p^2}{q}\gamma_p+q\gamma_q)+3(\frac{p}{q}\gamma_p^2+2\frac{p}{q}\gamma_p\gamma_q+\gamma_q^2).
\]
Taking the difference of the two terms, $\frac{\omega_p^{(4)}(u_0)}{p}=\frac{\omega_q^{(4)}(u_0)}{q}$, if and only if
\[
0=q^2-p^2+3\left(-p(\frac{p}{q}-1)\gamma_p-(q-p)\gamma_q+(\frac{p}{q}-1)\gamma_p^2+2(\frac{p}{q}-1)\gamma_p\gamma_q\right).
\]
Dividing by $3(1-\frac{p}{q})$, this necessary and sufficient condition becomes
\[
0=q\frac{p+q}{3}+p\gamma_p-q\gamma_q-\gamma_p^2-2\gamma_p\gamma_q,
\]
i.e.
\begin{equation}\label{eq:twogap1}
(2\gamma_p+q)\gamma_q
	=q\frac{p+q}{3}+p\gamma_p-\gamma_p^2.
\end{equation}
Fix $\gamma_p>0$ and $\gamma_q$ satisfying this latter equality. We get that $\gamma_q>0$ if and only if the left-hand side of the equality is positive, i.e.
\begin{equation}\label{eq:twogap2}
0< \gamma_p< \half\left(p+\sqrt{p^2+4q\frac{p+q}{3}}\right).
\end{equation}
Conversely, any two gap solution $u_0$ satisfying \eqref{eq:twogap1} and \eqref{eq:twogap2} verifies $\frac{\omega_p^{(4)}(u_0)}{p}=\frac{\omega_q^{(4)}(u_0)}{q}$, therefore is a traveling wave solution.
\end{proof}

Let us give an idea of the form of a two gap potential $u_0$ with gaps at indices $p<q$. By Theorem 3 in \cite{GerardKappeler2019}, the extension of $\Pi u_0$ as an holomorphic function on the unit disc $\{z\in\C\mid|z|<1\}$ satisfies
\[
\Pi u_0(z)=-z\frac{Q'(z)}{Q(z)}
\]
where $Q(z)=\det(\mathrm{Id}-zM)$ and $M=(M_{nm})_{0\leq n,m\leq q-1}$ is a $q\times q$ matrix defined by
\[
M_{nm}=
\begin{cases}
\delta_{m,n+1} &\text{ if } \zeta_{n+1}=0
\\
\sqrt{\mu_{n+1}}\sqrt{\frac{\kappa_m}{\kappa_{n+1}}}\frac{\zeta_m(u_0)\overline{\zeta_{n+1}(u_0)}}{(\lambda_m-\lambda_n-1)}  &\text{ if } \zeta_{n+1}\neq 0
\end{cases}.
\]
A precise definition of $\mu_{n}$, $\kappa_{n}$ and $\lambda_n$ can be found in \cite{GerardKappeler2019}.
Therefore, $Q$ and $Q'$ respectively write
\[
Q(z)=1-z^pM_{p-1,0}-z^qM_{q-1,0}M_{p-1,p}
\]
and
\[
Q'(z)=-pz^{p-1}M_{p-1,0}-qz^{q-1}M_{q-1,0}M_{p-1,p}
\]
This leads to
\[
u_0(x)
	=\frac{p\e^{ipx}M_{p-1,0}+q\e^{iqx}M_{q-1,0}M_{p-1,p}}{1-\e^{ipx} M_{p-1,0}-\e^{iqx}M_{q-1,0}M_{p-1,p}}
	+\frac{p\e^{-ipx}\overline{M_{p-1,0}}+q\e^{-iqx}\overline{M_{q-1,0}M_{p-1,p}}}{1-\e^{-ipx} \overline{M_{p-1,0}}-\e^{-iqx}\overline{M_{q-1,0}M_{p-1,p}}}.
\]

\begin{prop}
There are no three gap traveling waves.
\end{prop}

\begin{proof}
Let $p<q<r$ the indices for the nonzero gaps for a three gap potential $u_0$. The speeds for each mode write
\begin{multline*}
\frac{\omega_p^{(4)}(u_0)}{p}
	= p^2+(p\gamma_p+q\gamma_q+r\gamma_r)-3p(\gamma_p+\gamma_q+\gamma_r)\\
	+3(\gamma_p^2+2\gamma_p(\gamma_q+\gamma_r)+\gamma_q^2+2\gamma_q\gamma_r+\gamma_r^2),
\end{multline*}
\begin{multline*}
\frac{\omega_q^{(4)}(u_0)}{q}
	= q^2+(p\gamma_p+q\gamma_q+r\gamma_r)-3(\frac{p^2}{q}\gamma_p+q(\gamma_q+\gamma_r))\\
	+3(\frac{p}{q}\gamma_p^2+2\frac{p}{q}\gamma_p(\gamma_q+\gamma_r)+\gamma_q^2+2\gamma_q\gamma_r+\gamma_r^2)
\end{multline*}
and
\begin{multline*}
\frac{\omega_r^{(4)}(u_0)}{r}
	= r^2+(p\gamma_p+q\gamma_q+r\gamma_r)-3(\frac{p^2}{r}\gamma_p+\frac{q^2}{r}\gamma_q+r\gamma_r)\\
	+3(\frac{p}{r}\gamma_p^2+2\frac{p}{r}\gamma_p(\gamma_q+\gamma_r)+\frac{q}{r}\gamma_q^2+2\frac{q}{r}\gamma_q\gamma_r+\gamma_r^2).
\end{multline*}

Let us now subtract the equalities.
\begin{multline*}
\frac{\omega_q^{(4)}(u_0)}{q}-\frac{\omega_p^{(4)}(u_0)}{p}
	= q^2-p^2
	-3((\frac{p}{q}-1)p\gamma_p+(q-p)(\gamma_q+\gamma_r))\\
	+3((\frac{p}{q}-1)\gamma_p^2+2(\frac{p}{q}-1)\gamma_p(\gamma_q+\gamma_r)).
\end{multline*}
Dividing by $3(1-\frac{p}{q})$, we get that if $\frac{\omega_q^{(4)}(u_0)}{q}=\frac{\omega_p^{(4)}(u_0)}{p}$, then
\[
0
	=q\frac{p+q}{3}+p\gamma_p-q(\gamma_q+\gamma_r)-\gamma_p^2-2\gamma_p(\gamma_q+\gamma_r),
\]
or equivalently
\begin{equation}\label{eq:pq}
(q+2\gamma_p)(\gamma_q+\gamma_r)
	=q\frac{p+q}{3}+p\gamma_p-\gamma_p^2.
\end{equation}

Doing the same for indices $p$ and $r$,
\begin{multline*}
\frac{\omega_r^{(4)}(u_0)}{r}-\frac{\omega_p^{(4)}(u_0)}{p}
	= r^2-p^2
	-3((\frac{p}{r}-1)p\gamma_p+(\frac{q^2}{r}-p)\gamma_q+(r-p)\gamma_r)\\
	+3((\frac{p}{r}-1)\gamma_p^2+2(\frac{p}{r}-1)\gamma_p(\gamma_q+\gamma_r)+(\frac{q}{r}-1)\gamma_q^2+2(\frac{q}{r}-1)\gamma_q\gamma_r).
\end{multline*}

Dividing by $3(1-\frac{p}{r})$, if $\frac{\omega_r^{(4)}(u_0)}{r}=\frac{\omega_p^{(4)}(u_0)}{p}$, then
\begin{equation}\label{eq:pr}
0
	= r\frac{r+p}{3}
	+p\gamma_p+\frac{pr-q^2}{r-p}\gamma_q-r\gamma_r
	-\gamma_p^2-2\gamma_p(\gamma_q+\gamma_r)-\frac{r-q}{r-p}\gamma_q^2-2\frac{r-q}{r-p}\gamma_q\gamma_r.
\end{equation}

Subtracting \eqref{eq:pr} and \eqref{eq:pq}, if $\frac{\omega_p^{(4)}(u_0)}{p}=\frac{\omega_q^{(4)}(u_0)}{q}=\frac{\omega_r^{(4)}(u_0)}{r}$, then
\begin{equation*}
(r-q+2\frac{r-q}{r-p}\gamma_q)\gamma_r
	=\frac{r^2-q^2+p(r-q)}{3}+(\frac{pr-q^2}{r-p}+q)\gamma_q-\frac{r-q}{r-p}\gamma_q^2.
\end{equation*}
Since
\[
\frac{pr-q^2}{r-p}+q=\frac{pr-q^2+qr-pq}{r-p}=(p+q)\frac{r-q}{r-p},
\]
by multiplication by $\frac{r-p}{r-q}$, we get
\begin{equation}\label{eq:pr-pq}
(r-p+2\gamma_q)\gamma_r
	=(r-p)\frac{r+q+p}{3}+(p+q)\gamma_q-\gamma_q^2.
\end{equation}

Now, if $u_0$ is a traveling wave, the first equality \eqref{eq:pq} implies
\begin{align*}
\gamma_q+\gamma_r
	&=\frac{q\frac{p+q}{3}+p\gamma_p-\gamma_p^2}{q+2\gamma_p}\\
	&=\frac{p+q}{3}+\frac{-2\gamma_p\frac{p+q}{3}+p\gamma_p-\gamma_p^2}{q+2\gamma_p}\\
	&=\frac{p+q}{3}+\frac{\gamma_p}{3(q+2\gamma_p)}(p-2q-3\gamma_p),
\end{align*}
in particular, since $p<q$ and $\gamma_p>0$, necessarily
\begin{equation}\label{ineq:gammaq-gammar}
\gamma_q+\gamma_r<\frac{p+q}{3}.
\end{equation}

However, the second equality \eqref{eq:pr-pq} implies that
\begin{align*}
\gamma_r
	&=\frac{(r-p)\frac{r+q+p}{3}+(p+q)\gamma_q-\gamma_q^2}{r-p+2\gamma_q}\\
	&=\frac{(r-p)\frac{q+p}{3}+(p+q)\gamma_q-\gamma_q^2}{r-p+2\gamma_q}+\frac{(r-p)\frac{r}{3}}{r-p+2\gamma_q}\\
	&=\frac{p+q}{3}+\frac{-2\gamma_q\frac{p+q}{3}+(p+q)\gamma_q-\gamma_q^2}{r-p+2\gamma_q}+\frac{(r-p)\frac{r}{3}}{r-p+2\gamma_q}\\
	&=\frac{p+q}{3}+\gamma_q\frac{\frac{p+q}{3}-\gamma_q}{r-p+2\gamma_q}+\frac{(r-p)r}{3(r-p+2\gamma_q)}.\\
\end{align*}
Since from \eqref{ineq:gammaq-gammar},
\[
\gamma_q<\frac{p+q}{3},
\]
we get from this latter equality for $\gamma_r$ that
\[
\gamma_r>\frac{p+q}{3},
\]
but this is a contradiction with \eqref{ineq:gammaq-gammar}.
\end{proof}

\begin{cor}
There are no $N$ gap traveling waves for $N\geq 3$.
\end{cor}

\begin{proof}
The proof is the same as for the three gap traveling waves case, but with some additional terms which might hinder understanding for a first reading. We explain here how to adapt the proof.

Let $u_0$ be a $N$ gap potential, $N\geq 3$, and $p<q<r<r_4<\dots<r_N$ be the indices for the nonzero gaps. Let
\[
\Gamma_r:=\gamma_r+\sum_{k=4}^N\gamma_{r_k}.
\]
The speeds for the three smallest modes at indices $p,q$ and $r$ write
\begin{multline*}
\frac{\omega_p^{(4)}(u_0)}{p}
	= p^2+(p\gamma_p+q\gamma_q+r\gamma_r+\sum_{k=4}^Nr_k\gamma_{r_k})
	-3p(\gamma_p+\gamma_q+\Gamma_r)\\
	+3(\gamma_p^2+2\gamma_p(\gamma_q+\Gamma_r)+\gamma_q^2+2\gamma_q\Gamma_r+\gamma_r^2),
\end{multline*}
\begin{multline*}
\frac{\omega_q^{(4)}(u_0)}{q}
	= q^2+(p\gamma_p+q\gamma_q+r\gamma_r+\sum_{k=4}^Nr_k\gamma_{r_k})
	-3(\frac{p^2}{q}\gamma_p+q(\gamma_q+\Gamma_r))\\
	+3(\frac{p}{q}\gamma_p^2+2\frac{p}{q}\gamma_p(\gamma_q+\Gamma_r)+\gamma_q^2+2\gamma_q\Gamma_r+\Gamma_r^2)
\end{multline*}
and
\begin{multline*}
\frac{\omega_r^{(4)}(u_0)}{r}
	= r^2+(p\gamma_p+q\gamma_q+r\gamma_r+\sum_{k=4}^Nr_k\gamma_{r_k})
	-3(\frac{p^2}{r}\gamma_p+\frac{q^2}{r}\gamma_q+r\Gamma_r)\\
	+3(\frac{p}{r}\gamma_p^2+2\frac{p}{r}\gamma_p(\gamma_q+\Gamma_r)+\frac{q}{r}\gamma_q^2+2\frac{q}{r}\gamma_q\Gamma_r+\Gamma_r^2).
\end{multline*}
The rest of the proof is identical up to replacing $\gamma_r$ by $\Gamma_r$ everywhere from this point on.
\end{proof}

\subsection{Orbital stability}\label{subsection:stability}

\begin{prop}
The one gap traveling waves in $L^2_{r,0}(\T)$ are orbitally stable.
\end{prop}

\begin{proof}
Let $u_0$ be a one gap traveling wave, $u_0^k$ a sequence of initial data converging to $u_0$ in $L^2_{r,0}(\T)$, and $t_k$ a sequence of times. We prove that up to some subsequence,
\[
\inf_{\theta\in\T}\|u_k(t_k)-u_0(\cdot+\theta)\|_{L^2_{r,0}(\T)}
	\longrightarroww{k\to+\infty}{} 0.
\]
It is enough to show that there exists $\theta\in\T$ such that in $L^2_{r,0}(\T)$,
\[
u_k(t_k)\longrightarroww{k\to+\infty}{} u_0(\cdot+\theta),
\]
i.e.\@ such that in $h^{\half}_+$,
\[
\Phi(u_k(t_k))\longrightarroww{k\to+\infty}{} \Phi(u_0(\cdot+\theta)).
\]

Recall that
\[
\Phi(u_k(t_k))=(\zeta_n(u_0^k)\e^{i\omega_n^{(4)}(u_0^k)t_k})_n.
\]
Let $p\geq1$ be the index for which $\zeta_p(u_0)\neq 0$. Up to some subsequence, there exists $\theta\in\T$ such that $\e^{i\omega_p^{(4)}(u_0^k)t_k}\longrightarroww{k\to+\infty}{} \e^{i\theta}$. Moreover, since $u_0^k$ converges to $u_0$ in $L^2_{r,0}(\T)$,
\[
p|\zeta_p(u_0^k)-\zeta_p(u_0)|^2+\sum_{n\neq p} n|\zeta_n(u_0^k)|^2
	\longrightarroww{k\to+\infty}{} 0.
\]
We deduce that
\begin{align*}
\|\Phi(u_k(t_k))- \Phi(u_0(\cdot+\theta))\|_{h^{\half}_+}^2
	&= p|\zeta_p(u_0^k)\e^{i\omega_n^{(4)}(u_0^k)t_k}-\zeta_p(u_0)\e^{i\theta}|^2+\sum_{n\neq p} n|\zeta_n(u_0^k)|^2\\
	&\longrightarroww{k\to+\infty}{} 0.
\end{align*}
\end{proof}

\begin{prop}
The two gap traveling waves in $L^2_{r,0}(\T)$ are orbitally unstable.
\end{prop}

\begin{proof}
Let $u_0$ be a two gap traveling wave such that the nonzero terms of the sequence $\Phi(u_0)=(\zeta_n(u_0))_{n\geq 1}$ are $\zeta_p(u_0)$ and $\zeta_q(u_0)$. We define the sequence $u_0^k$ of two gap initial data by their nonzero gaps at indices $p$ and $q$, denoted $\zeta_p(u_0^k)$ and $\zeta_q(u_0^k)$, as follows. We fix $\zeta_p(u_0^k):=\zeta_p(u_0)$ and choose any sequence of nonzero complex numbers $(\zeta_q(u_0^k))_k$ such that $\zeta_q(u_0^k)\longrightarroww{k\to+\infty}{}\zeta_q(u_0)$ but for all $k\in\N$, $\varepsilon_k:=|\zeta_q(u_0^k)|^2-|\zeta_q|^2\neq 0$. Then we construct $t_k\in\R$ in order to negate the orbital stability of $u_0$.

Assume by contradiction that
\[
\inf_{\theta\in\T}\|u_k(t_k)-u_0(\cdot+\theta)\|_{L^2_{r,0}(\T)}\longrightarroww{k\to+\infty}{}0.
\]
Then there exists a sequence $\theta_k\in\T$, $k\in\N$, such that in $L^2_{r,0}(\T)$,
\[
\|u_k(t_k,\cdot-\theta_k)-u_0\|_{L^2_{r,0}(\T)}\longrightarroww{k\to+\infty}{} 0.
\]
Applying the Birkhoff map, which is continuous on $L^2_{r,0}(\T)$,
\[
\zeta_p(u_0^k)\e^{i\omega^{(4)}_p(u_0^k)t_k-ip\theta_k}
	=\zeta_p(u_k(t_k))\e^{-ip\theta_k}
	\longrightarroww{k\to+\infty}{}\zeta_p(u_0)
\]
and
\[
\zeta_q(u_0^k)\e^{i\omega^{(4)}_q(u_0^k)t_k-iq\theta_k}
	=\zeta_q(u_k(t_k))\e^{-iq\theta_k}
	\longrightarroww{k\to+\infty}{}\zeta_q(u_0).
\]
This implies by taking the arguments that for some integers $n_{p,k}$ and $n_{q,k}$,
\[
\omega^{(4)}_p(u_0^k)t_k-p\theta_k+2\pi n_{p,k}\longrightarroww{k\to+\infty}{} 0
\]
and
\[
\omega^{(4)}_q(u_0^k)t_k-q\theta_k+2\pi n_{q,k}\longrightarroww{k\to+\infty}{} 0,
\]
therefore
\begin{equation}\label{eq:limit}
pqt_k\left(\frac{\omega^{(4)}_q(u_0^k)}{q}
-\frac{\omega^{(4)}_p(u_0^k)}{p}\right) +2\pi(pn_{q,k}-qn_{p,k})\longrightarroww{k\to+\infty}{} 0.
\end{equation}

However, writing $\varepsilon_k=|\zeta_q(u_0^k)|^2-|\zeta_q(u_0)|^2=\gamma_q(u_0^k)-\gamma_q(u_0)$, we get that the speeds of the two modes $p$ and $q$ for the initial data $u_0^k$ are given by
\[
\frac{\omega^{(4)}_p(u_0^k)}{p}
	=\frac{\omega^{(4)}_p(u_0)}{p}+q\varepsilon_k
	-3p\varepsilon_k+6\varepsilon_k(\gamma_p(u_0)+\gamma_q(u_0))+3\varepsilon_k^2
\]
and
\[
\frac{\omega^{(4)}_q(u_0^k)}{q}
	=\frac{\omega^{(4)}_q(u_0)}{q}+q\varepsilon_k
	-3q\varepsilon_k+6\varepsilon_k(\frac{p}{q}\gamma_p(u_0)+\gamma_q(u_0))+3\varepsilon_k^2.
\]
Since $u_0$ is a traveling wave, $\frac{\omega^{(4)}_p(u_0)}{p}=\frac{\omega^{(4)}_q(u_0)}{q}$ and therefore
\[
\frac{\omega^{(4)}_q(u_0^k)}{q}-\frac{\omega^{(4)}_p(u_0^k)}{p}
	=-3(q-p+2(1-\frac{p}{q})\gamma_p(u_0))\varepsilon_k.
\]
Since $\varepsilon_k\neq 0$, then $\frac{\omega^{(4)}_q(u_0^k)}{q}\neq \frac{\omega^{(4)}_p(u_0^k)}{p}$. It is therefore possible to choose a sequence $t_k$ such that the limit \eqref{eq:limit} does not hold and get a contradiction. For instance, we can choose $t_k$ such that
\[
-3pq t_k(q-p+2(1-\frac{p}{q})\gamma_p)\varepsilon_k =\pi.
\]
\end{proof}


\appendix
\section{Appendices}

\subsection{About the hierarchy}\label{part:hierarchy}

The aim of this Appendix is to provide a way to compute the Hamiltonians $\hamilton_k$ and frequencies $\omega_n^{(k)}$ for the higher order Benjamin-Ono equations in terms of the actions $\gamma_p=|\zeta_p(u_0)|^2$. In particular, we establish formula \eqref{eq:omega_n^4}
\begin{align*}
\omega_n^{(4)}(u_0)
	&=n^3+n\sum_{p\geq 1}p\gamma_p-3\sum_{p\geq1}\min(p,n)^2\gamma_p+3\sum_{p,q\geq 1}\min(p,q,n)\gamma_p\gamma_q
\end{align*}
for finite gap potentials $u_0$.


We need to recall first some notation.
Given $u\in L^2_{r,0}(\T)$ , we consider its Lax operator $L_u=-i\partial_x-T_u$ acting on $L^2_+(\T)$, with domain $H^1_+(\T)=H^1(\T)\cap L^2_+(\T)$. The spectrum of $L_u$ is discrete, with eigenvalues
\[
\lambda_0(u)<\lambda_1(u)<\cdots<\lambda_n(u)<\cdots.
\]
Moreover (see \cite{GerardKappeler2019}),
\[
\gamma_n(u)=\lambda_n(u)-\lambda_{n-1}(u)-1, \quad n\geq 1
\]
is non-negative and satisfies $\gamma_n(u)=|\zeta_n(u)|^2$. We also define $f_n(u)\in H^1_+(\T)$ as the $L^2$-normalized eigenfunction for $L_u$ associated to the eigenvalue $\lambda_n(u)$.

Let $u_0$ be a finite gap potential and use the above notation. From \cite{GerardKappeler2019} (3.8), a variant of the generating function, denoted by $\widetilde{\mathcal{H}_\varepsilon}$, is defined as
\[
\widetilde{\mathcal{H}_\varepsilon}=\sum_{n=0}^{+\infty} \frac{|\langle \un|f_n\rangle|^2}{1+\varepsilon\lambda_n}.
\]
From the decomposition (2.12) in \cite{GerardKappeler2019} :
\(
\Pi u=-\sum_{n=1}^{+\infty}\lambda_n\langle \un|f_n\rangle f_n,
\)
we know using formula \eqref{eq:Hk_Lax} that
\begin{align*}
\hamilton_k(u)
	&=\sum_{n=0}^{+\infty} |\langle \un|f_n\rangle|^2\lambda_n^k\\
	&=\frac{(-1)^k}{k!}\frac{\d^k}{\d\varepsilon^k}|_{\varepsilon=0}\widetilde{\mathcal{H}_\varepsilon}.
\end{align*}


We now make use of the generating function to derive a recurrence formula for the $\hamilton_k$. Set
\begin{align*}
g_\varepsilon
	:=-\frac{\d}{\d\varepsilon}\log\widetilde{\mathcal{H}_\varepsilon}
	=-\frac{1}{\widetilde{\mathcal{H}_\varepsilon}}\frac{\d}{\d\varepsilon}\widetilde{\mathcal{H}_\varepsilon},
\end{align*}
then from \cite{GerardKappeler2019} (3.11), $g_\varepsilon$ writes
\begin{align*}
g_\varepsilon
	=\frac{\lambda_0}{1+\varepsilon\lambda_0}+\sum_{n=1}^{+\infty}\frac{\gamma_n}{(1+\varepsilon(\lambda_{n-1}+1))(1+\varepsilon\lambda_n)}. 
\end{align*}
Using the identity
\begin{align*}
-\frac{\d^{k+1}}{\d\varepsilon^{k+1}}\widetilde{\mathcal{H}_\varepsilon}
	&=\frac{\d^k}{\d\varepsilon^k}\left(g_\varepsilon\widetilde{\mathcal{H}_\varepsilon}\right)\\
	&=\sum_{l=0}^k\binom{k}{l}\frac{\d^l}{\d\varepsilon^l}(g_\varepsilon)\frac{\d^{k-l}}{\d\varepsilon^{k-l}}(\widetilde{\mathcal{H}_\varepsilon}).
\end{align*}
and defining
\[
P_l:=\frac{(-1)^l}{l!}\frac{\d^l}{\d\varepsilon^l}|_{\varepsilon=0}(g_\varepsilon), \]
we get the recurrence relation
\begin{equation}\label{eq:fla_hk}
\hamilton_{k+1}=\frac{1}{k+1}\sum_{l=0}^kP_l\hamilton_{k-l}.
\end{equation}
Moreover, the frequencies $\omega_n^{(k)}=\frac{\partial \hamilton_k}{\partial \gamma_n}$, satisfy the recurrence formula
\begin{equation}\label{eq:fla_omegak}
\omega_n^{(k+1)}=\frac{1}{k+1}\sum_{l=0}^k\frac{\partial P_l}{\partial \gamma_n}\hamilton_{k-l}+P_l \omega_n^{(k-l)}.
\end{equation}

We now simplify $P_l$ and $\frac{\partial P_l}{\partial \gamma_n}$ :
\begin{align*}
P_l
	&=\lambda_0^{l+1}+\sum_{n\geq 1}\gamma_n\sum_{m=0}^l(\lambda_{n-1}+1)^m\lambda_n^{l-m},
\end{align*}
and since $\lambda_{n-1}+1=\lambda_n-\gamma_n$,
\begin{align*}
P_l
	&=\lambda_0^{l+1}+\sum_{n\geq 1}\lambda_n^{l+1}-(\lambda_n-\gamma_n)^{l+1}.
\end{align*}
From (3.13), $\lambda_n=n-s_{n+1}$ where $s_n=\sum_{k=n}^\infty\gamma_n$ for $n\geq 1$, therefore
\begin{align}\label{eq:pk}
P_l
	&=(-1)^{l+1}s_1^{l+1}+\sum_{n\geq 1}(n-s_{n+1})^{l+1}-(n-s_n)^{l+1}.
\end{align}
We deduce
\begin{align*}
\frac{1}{l+1}\frac{\partial P_l}{\partial \gamma_n}
	&=(-1)^{l+1}s_1^l+(n-s_n)^l-\sum_{p<n}(p-s_{p+1})^l-(p-s_p)^l\\
	&=\sum_{p=1}^n(p-s_p)^l-(p-1-s_p)^l\\
	&=n^l+\sum_{p=1}^n\sum_{m=1}^{l-1}\binom{l}{m}(p^m-(p-1)^m)(-1)^{l-m}s_p^{l-m}\\
	&=n^l+\sum_{m=1}^{l-1}\binom{l}{m}(-1)^{l-m}\sum_{\substack{p,p_1,\dots,p_{l-m}\\1\leq p\leq \min(n,p_1,\dots,p_{l-m})}}(p^m-(p-1)^m)\gamma_{p_1}\dots\gamma_{p_{l-m}},
\end{align*}
therefore
\begin{align*}
\frac{1}{l+1}\frac{\partial P_l}{\partial \gamma_n}
	&=n^l+\sum_{m=1}^{l-1}\binom{l}{m}(-1)^{l-m}\sum_{p_1,\dots,p_{l-m}}\min(n,p_1,\dots,p_{l-m})^m\gamma_{p_1}\dots\gamma_{p_{l-m}}.
\end{align*}


Let us compute the first small terms by using \eqref{eq:fla_hk} and \eqref{eq:fla_omegak}.
The Hamiltonian with index $0$ is constant
\(
\hamilton_0=1,
\)
leading to
\(
\omega_n^{(0)}=0,
\)
\(
P_0=0
\)
and
\(
\frac{\partial P_0}{\partial \gamma_n}=0.
\)

Concerning index $1$,
\(
\hamilton_1=-\langle u|\un\rangle=0
\)
because we assumed that $u\in L^2_{r,0}(\T)$ is of average zero. This leads to
\(
\omega_n^{(1)}=0,
\)
\(
P_1
	=2\sum_{p\geq 1}p\gamma_p
\)
and
\(
\frac{\partial P_1}{\partial \gamma_n}=2n.
\)
Because of cancellations for several of these small terms, the recurrence relations \eqref{eq:fla_hk} and \eqref{eq:fla_omegak} write, for $k\geq 2$,
\begin{equation}\label{eq:fla_hk2}
\hamilton_{k}=\frac{1}{k}\left(P_1\hamilton_{k-2}+P_2\hamilton_{k-3}+\cdots+P_{k-3}\hamilton_2+P_{k-1}\right)
\end{equation}
and
\begin{multline}\label{eq:fla_omegak2}
\omega^{(k)}_n=\frac{1}{k}\Big(\frac{\partial P_1}{\partial {\gamma_n}}\hamilton_{k-2}+\frac{\partial P_2}{\partial {\gamma_n}}\hamilton_{k-3}+\cdots+\frac{\partial P_{k-3}}{\partial {\gamma_n}}\hamilton_2\\
	+P_1\omega^{(k-2)}_n+P_2\omega^{(k-3)}_n+\cdots+P_{k-3}\omega^{(2)}_n	
	+\frac{\partial P_{k-1}}{\partial {\gamma_n}}\Big).
\end{multline}

The second index leads to the conservation of the mass
\[
\hamilton_2=\frac{\|u\|^2}{2}=\frac{P_1}{2},
\]
\[
\omega_n^{(2)}=n.
\]
Moreover,
\[
P_2
	=-s_1^3+\sum_{p=1}^\infty(p-s_{p+1})^3-(p-s_p)^3
	=3\sum_{p\geq1}p^2\gamma_p-3\sum_{p\geq1}s_p^2
\]
and
\[
\frac{\partial P_2}{\partial \gamma_n}
	=3(n^2-2\sum_{p=1}^\infty\min(p,n)\gamma_p).
\]

For the third index, we retrieve identity (8.6) from \cite{GerardKappeler2019} for the Hamiltonian
\[
\hamilton_3=\frac{P_2}{3}=\sum_{p\geq1}p^2\gamma_p-\sum_{p\geq1}s_p^2
\]
and formula (8.4) from \cite{GerardKappeler2019} for the frequencies
\[
\omega_n^{(3)}=n^2-2\sum_{p=1}^{+\infty}\min(p,n)\gamma_p.
\]

We now use that
\[
P_3=s_1^4+\sum_{p=1}^{+\infty}(p-s_{p+1})^4-(p-s_p)^4
\]
and
\begin{align*}
\frac{\partial P_3}{\partial \gamma_n}
	&=4(n^3-3\sum_{p}\min(p,n)^2\gamma_p+3\sum_{p,q}\min(p,q,n)\gamma_p\gamma_q)\\
\end{align*}
to get the formula for the Hamiltonian of index $4$
\[
\hamilton_4=\frac{1}{4}( P_3+ P_1\hamilton_2) 
	=\frac{1}{4} P_3+\half \hamilton_2^2
\]
and the frequencies
\begin{align*}
\omega_n^{(4)}
	&=\frac{1}{4}\frac{\partial P_3}{\partial \gamma_n}+n \hamilton_2\\
	&=n^3+n\sum_{p\geq 1}p\gamma_p-3\sum_{p\geq1}\min(p,n)^2\gamma_p+3\sum_{p,q\geq 1}\min(p,q,n)\gamma_p\gamma_q.
\end{align*}

In the same way, using that
\[
P_4=-s_1^5+\sum_{p=1}^{+\infty}(p-s_{p+1})^5-(p-s_p)^5
\]
and
\begin{align*}
\frac{1}{5}\frac{\partial P_4}{\partial \gamma_n}	
	&=n^4-4\sum_{p\geq1}\min(p,n)^3\gamma_p+6\sum_{p,q\geq1}\min(p,q,n)^2\gamma_p\gamma_q-4\sum_{p,q,r\geq 1}\min(p,q,r,n)\gamma_p\gamma_q\gamma_r,
\end{align*}
we can get a formula for the Hamiltonian with index $5$
\[
\hamilton_5=\frac{1}{5}(P_4+P_2\hamilton_2+P_1\hamilton_3)
\]
and the frequencies
\begin{align*}
\omega_n^{(5)}
	&=\frac{1}{5}(\frac{\partial P_1}{\partial \gamma_n}\hamilton_3+P_1\omega_n^{(3)}+\frac{\partial P_2}{\partial \gamma_n}\hamilton_2+P_2\omega_n^{(2)}+\frac{\partial P_4}{\partial \gamma_n})\\
	&=\frac{2}{5}n(\sum_{p\geq 1}p^2\gamma_p-\sum_{p\geq1} s_p^2)+\frac{2}{5}(\sum_{p\geq 1}p\gamma_p)(n^2-2\sum_{p\geq 1}\min(p,n)\gamma_p)\\
	&+\frac{3}{5}(n^2-2\sum_{p\geq 1}\min(p,n)\gamma_p)(\sum_{p\geq 1}p\gamma_p)
	+\frac{3}{5}(\sum_{p\geq 1}p^2\gamma_p-\sum_{p\geq1} s_p^2)n
	+\frac{1}{5}\frac{\partial P_4}{\partial \gamma_n},
\end{align*}
leading to
\begin{multline}\label{eq:omega_n^5}
\omega_n^{(5)}
	=n(\sum_{p\geq 1}p^2\gamma_p-\sum_{p\geq1} s_p^2)+(\sum_{p\geq 1}p\gamma_p)(n^2-2\sum_{p\geq 1}\min(p,n)\gamma_p)\\
	+n^4-4\sum_{p\geq1}\min(p,n)^3\gamma_p+6\sum_{p,q\geq1}\min(p,q,n)^2\gamma_p\gamma_q-4\sum_{p,q,r\geq 1}\min(p,q,r,n)\gamma_p\gamma_q\gamma_r.
\end{multline}

\begin{rk}\label{rk:WP}
Note that formulas \eqref{eq:omega_n^4} and \eqref{eq:omega_n^5} for $\omega_n^{(4)}(u_0)$ and $\omega_n^{(5)}(u_0)$, which have been established for finite gap potentials $u_0$, still make sense for $\omega_n^{(4)}(u_0)$ if $u_0\in L^2_{r,0}(\T)$ and for $\omega_n^{(5)}(u_0)$ if $u_0\in H^{\half}_{r,0}(\T)$, but diverge if $u_0\in H^s_{r,0}(\T)$ for $s$ right below these respective exponents ($s<0$ and $s<\half$).

One can actually show by induction the following facts. In the formula \eqref{eq:pk} for $P_k$, there is one term $c_1\sum_{p=1}^{+\infty}p^k\gamma_p$, the other terms being convergent if $\sum_{p=1}^{+\infty}p^{k-1}\gamma_p<+\infty$. This implies that in the formula \eqref{eq:fla_hk2} for $\hamilton_k$ appears one term $c_2\sum_{p=1}^{+\infty}p^{k-1}\gamma_p$, the other terms being convergent if $\sum_{p=1}^{+\infty}p^{k-2}\gamma_p<+\infty$. Consequently, in formula \eqref{eq:fla_omegak2} for $\omega^{(k)}_n$, $k\geq 4$, appears one term $c_3\sum_{p=1}^{+\infty}p^{k-3}\gamma_p$, the other terms being convergent if $\sum_{p=1}^{+\infty}p^{k-4}\gamma_p<+\infty$.

From these facts, one can see that the formula for $\omega^{(k)}_n$ can be extended by continuity to potentials in $H^{s_k}_{r,0}(\T)$ where $s_k=\frac{k}{2}-2$, however there is no continuous extension to $H^{s_k}_{r,0}(\T)$ when $s_k-\half<s<s_k$. This explains why the well-posedness threshold for the equation associated to the Hamiltonian $\hamilton_k$ in the hierarchy should be $H^{s_k}_{r,0}(\T)$.
\end{rk}

\subsection{Equation for the fourth Hamiltonian}\label{part:appendix}

From formula \eqref{eq:Hk_Lax} and the decomposition (2.12) in \cite{GerardKappeler2019} :
\(
\Pi u=-\sum_{n=1}^{+\infty}\lambda_n\langle \un|f_n\rangle f_n,
\) we see that for $k\geq 2$,
\begin{equation}\label{eq:Hk_Lax2}
\hamilton_k(u)=\langle L_u^{k-2}\Pi u|\Pi u\rangle,
\end{equation}
where
\[
L_u(h)=Dh-\Pi(uh), \quad D=-i\partial_x, \quad h\in H^1_+(\T).
\]

For instance,
\[
\hamilton_3(u)=\frac{1}{2\pi}\int_0^{2\pi}\half u\hilbert\partial_xu-\frac{1}{3}u^3\d x
\]
leads to the Benjamin-Ono equation
\[
\partial_t u=\hilbert \partial_x^2u-\partial_x(u^2).
\]




\begin{prop}
The Hamiltonian for the third order equation of the Benjamin-Ono hierarchy~\eqref{eq:bo4} is
\[
\hamilton_4(u)=\intt \left( \half (\partial_x u)^2-\frac{3}{4}u^2\hilbert\partial_xu+\frac{1}{4}u^4\right)-\frac{1}{8}\|u\|_{L^2(\T)}^4,
\]
therefore the third order equation of the Benjamin-Ono hierarchy writes
\[
\partial_tu=\partial_x(-\partial_{xx}u-\frac{3}{2}u\hilbert\partial_x u-\frac{3}{2}\hilbert(u\partial_x u)+u^3)
\]
\end{prop}

\begin{proof}
Let $u_0\in L^2_{r,0}(\T)$. We develop
\begin{align*}
\hamilton_4(u)
	=\|D\Pi u\|_{L^2(\T)}^2-2\Re\langle D\Pi u|\Pi(u\Pi u))\rangle+\|\Pi(u\Pi u)\|_{L^2(\T)}^2,
\end{align*}
and study each term separately.

First, since $u$ is real, $\widehat{u}(-n)=\overline{\widehat{u}(n)}$, therefore
\begin{align*}
\|D\Pi u\|_{L^2(\T)}^2
	=\sum_{n\geq 0}|n|^2|\widehat{u}(n)|^2
	=\half\sum_{n\in\Z}|n|^2|\widehat{u}(n)|^2
	=\half\|\partial_x u\|_{L^2(\T)}^2.
\end{align*}

Then, $u$ being with average zero, $u=\Pi u+\overline{\Pi u}$, leading to
\begin{align*}
\langle D\Pi u|\Pi(u\Pi u))\rangle
	&=\intt D(\Pi u) u \overline{\Pi u}\d x\\
	&=\intt D( u) u \overline{\Pi u}\d x-\intt D(\overline{\Pi u}) u \overline{\Pi u}\d x
\end{align*}
so that
\begin{align*}
2\langle D\Pi u|\Pi(u\Pi u))\rangle
	&=\intt D( u^2)\overline{\Pi u}\d x-\intt D(\overline{\Pi u}^2) u\d x\\
	&=-\intt  u^2D(\overline{\Pi u})\d x+\intt \overline{\Pi u}^2 Du\d x.
\end{align*}
Taking the real part,
\begin{align*}
4\Re\langle D\Pi u|\Pi(u\Pi u))\rangle
	&=2\left(\langle D\Pi u|\Pi(u\Pi u))\rangle+\overline{\langle D\Pi u|\Pi(u\Pi u))\rangle}\right)\\
	&=-\intt  u^2(D(\overline{\Pi u})+\overline{D\overline{\Pi u}})\d x+\intt \overline{\Pi u}^2 Du+(\Pi u)^2\overline{Du}\d x.
\end{align*}
Using that $\overline{Df}=-D\overline{f}$,
\begin{align*}
4\Re\langle D\Pi u|\Pi(u\Pi u))\rangle
	&=-\intt  u^2(D(\overline{\Pi u})-D(\Pi u))\d x+\intt (\overline{\Pi u}^2 -(\Pi u)^2)Du\d x\\
	&=-\intt  u^2(D(\overline{\Pi u})-D(\Pi u))\d x+\intt (\overline{\Pi u} -\Pi u)uDu\d x\\
	&=-\intt  u^2(D(\overline{\Pi u})-D(\Pi u))\d x-\half \intt D(\overline{\Pi u} -\Pi u)u^2\d x\\
	&=-\frac{3}{2}\intt  u^2(D(\overline{\Pi u})-D(\Pi u))\d x.
\end{align*}
It now remains to remark that $D(\overline{\Pi u})-D(\Pi u)=-H\partial_xu$ in order to conclude the identity
\[
2\Re\langle D\Pi u|\Pi(u\Pi u))\rangle
	=\frac{3}{4}\langle H\partial_xu|u^2\rangle.
\]

Finally, we treat the last term $\|\Pi(u\Pi u)\|_{L^2(\T)}^2$. Note that by decomposing $u=\Pi u+\overline{\Pi u}$,
\[
\Pi(u\Pi u)=(\Pi u)^2+\Pi(\overline{\Pi u}\Pi u),
\]
therefore
\begin{align*}
|\Pi(u\Pi u)|^2
	&=|\Pi u|^4+(\Pi u)^2\overline{\Pi(\overline{\Pi u}\Pi u)}+\overline{\Pi u}^2\Pi(\overline{\Pi u}\Pi u)+|\Pi(\overline{\Pi u}\Pi u)|^2.
\end{align*}
By removing the useless projections,
\begin{align*}
\|\Pi(u\Pi u)\|_{L^2(\T)}^2
	&=\intt |\Pi u|^4+(\Pi u)^2\overline{\Pi u}\Pi u+\overline{\Pi u}^2\overline{\Pi u}\Pi u+|\Pi(\overline{\Pi u}\Pi u)|^2\d x\\
	&=\intt (\Pi u)^2\overline{\Pi u}^2+(\Pi u)^3\overline{\Pi u}+\overline{\Pi u}^3\Pi u+|\Pi(\overline{\Pi u}\Pi u)|^2\d x.
\end{align*}
But if we take the fourth power of the identity $u=\Pi u+\overline{\Pi u}$
\[
u^4=(\Pi u)^4+\overline{\Pi u}^4+4(\Pi u)^3\overline{\Pi u}+4\overline{\Pi u}^3\Pi u+6\overline{\Pi u}^2(\Pi u)^2,
\]
and make use of the fact that the mean of $u$ is zero, we get
\[
\|u\|_{L^4(\T)}^4
	=\intt 4(\Pi u)^3\overline{\Pi u}+4\overline{\Pi u}^3\Pi u+6(\Pi u)^2\overline{\Pi u}^2\d x.
\]
By subtraction, the following cancellations happen :
\begin{align*}
\|\Pi(u\Pi u)\|_{L^2(\T)}^2-\frac{1}{4}\|u\|_{L^4(\T)}^4
	&=\intt |\Pi(\overline{\Pi u}\Pi u)|^2\d x-\half\intt (\Pi u)^2\overline{\Pi u}^2\d x.	
\end{align*}
To conclude, since $\overline{\Pi u}\Pi u$ is real, we can use the identity $\|f\|^2+|\langle f |\un\rangle|^2=2\|\Pi f\|^2$ for real valued functions $f\in L^2(\T)$ to get
\begin{align*}
\intt |\Pi(\overline{\Pi u}\Pi u)|^2\d x
	=\half\intt |\overline{\Pi u}\Pi u|^2\d x-\half |\langle \overline{\Pi u}\Pi u,\un\rangle|^2,
\end{align*}
leading to
\begin{align*}
\|\Pi(u\Pi u)\|_{L^2(\T)}^2-\frac{1}{4}\|u\|_{L^4(\T)}^4
	&=-\half |\langle \overline{\Pi u}\Pi u,\un\rangle|^2\\
	&=-\half\|\Pi u\|_{L^2(\T)}^4\\
	&=-\frac{1}{8}\|u\|_{L^2(\T)}^4.
\end{align*}

\end{proof}

\subsection{Structure of the higher order Hamiltonians}\label{part:appendix_structure}

The aim of this Appendix is to give an alternative proof of Proposition 2.2 in \cite{TzvetkovVisciglia2014}.

We first recall the notation introduced in \cite{TzvetkovVisciglia2014} for the sake of completeness. For a smooth function $u\in\classeC^{\infty}(\T)$, define by induction the sets $\mathcal{P}_n(u)$ as
\[
\mathcal{P}_1(u)
	=\{H^{\varepsilon_1}\partial_x^{\alpha_1}u
	\mid \varepsilon_1\in\{0,1\},\quad \alpha_1\in\N\},
\]
\[
\mathcal{P}_2(u)
	=\{(H^{\varepsilon_1}\partial_x^{\alpha_1}u)(H^{\varepsilon_2}\partial_x^{\alpha_2}u)
	\mid \varepsilon_1,\varepsilon_2\in\{0,1\},\quad \alpha_1,\alpha_2\in\N\}
\]
and for $n\geq 2$,
\[
\mathcal{P}_n(u)
	=\left\{\prod_{l=1}^kH^{\varepsilon_l}p_{j_l}(u)
	\mid
	k\in\li 2,n\ri,
	\quad \varepsilon_1,\dots, \varepsilon_k\in\{0,1\},
	\quad \sum_{l=1}^kj_l=n,
	\quad p_{j_l}(u)\in\mathcal{P}_{j_l}(u)\right\}.
\]
Moreover, for $p_n(u)\in\mathcal{P}_n(u)$, the term $\widetilde{p_n}(u)$ is uniquely defined from $p_n(u)$ by removing all the symbols $H$ in the expression of $p_n(u)$ and only keeping the symbols $\partial_x^{\alpha_i}u$. In this case, if
\[
\widetilde{p_n}(u)=\prod_{i=1}^n\partial_x^{\alpha_i}u,
\]
the maximal order of derivative involved and the sum of these orders are respectively denoted
\[
|p_n(u)|=\sup_{i\in\li 1,n\ri}\alpha_i
\]
and
\[
\|p_n(u)\|=\sum_{i=1}^n\alpha_i.
\]

We now retrieve a proof of the following result (Proposition 2.2 in \cite{TzvetkovVisciglia2014}).
\begin{prop}
Let $k=2(m+1)$ be an even integer. Then there exists $c\in\R$ such that the $k$-th Hamiltonian $\hamilton_k$ writes, for all $u\in\classeC^{\infty}(\T)$,
\[
\hamilton_{k+2}(u)
	=\half\|u\|_{\dot{H}^{m+1}(\T)}^2
	+c\int_0^{2\pi} u(H\partial_x^mu)(\partial_x^{m+1}u)\d x
	+R,
\]
where for some real numbers $c(p)$,
\[
R=\sum_{j=3}^{2m+4}\sum_{\substack{p(u)\in\mathcal{P}_j(u)\\\|p(u)\|=2m+4-j\\|p(u)|\leq m}}c(p)\int_0^{2\pi}p(u)\d x.
\]
\end{prop}
Note that $\hamilton_{k+2}(u)=\half E_{k/2}(u)$ with the notation from \cite{TzvetkovVisciglia2014}.

\begin{proof}
Recall formula \eqref{eq:Hk_Lax2}
\begin{align*}
\hamilton_{k+2}(u)
	&=\langle L_u^{k}\Pi u|\Pi u\rangle\\
	&=\langle L_u^{m+1}\Pi u| L_u^{m+1}\Pi u\rangle\\
	&=\langle (D-T_u)^{m+1}\Pi u| (D-T_u)^{m+1}\Pi u\rangle,
\end{align*}
where
$D=-i\partial_x$ and $T_u:h\in L^2_+(\T)\mapsto\Pi(uh)$.

We expand $\hamilton_{k+2}(u)$ as a sum of terms depending on whether we applied the operator $D$ or the operator $T_u$ when applying $L_u$.

It is possible to decompose $\hamilton_{k+2}(u)$ as follows :
\[
\hamilton_{k+2}(u)
	=\half\|u\|_{\dot{H}^{m+1}(\T)}^2
	+A+B,
\]
where $\|\Pi u\|_{\dot{H}^{m+1}(\T)}^2=\half\|u\|_{\dot{H}^{m+1}(\T)}^2$ is obtained when one only applies operator $D$, $A$ is obtained when one applies only once the operator $T_u$ and $(2m+1)$ times the operator $D$
\begin{align*}
A
	&=-2\Re\left(\sum_{j=0}^m\langle D^{m-j}\Pi(u D^j\Pi u)|D^{m+1}\Pi u\rangle\right)\\
	&=-2\Re\left(\sum_{j=0}^m\langle D^{m-j}(u D^j\Pi u)|D^{m+1}\Pi u\rangle\right),
\end{align*}
and $B$ is obtained when we apply at least twice in total the operator $T_u$.

$\bullet$ We first prove that one can decompose $A$ as
\[
A
	=c\int_0^{2\pi} u(H\partial_x^mu)(\partial_x^{m+1}u)\d x
	+\widetilde{A}
\]
where for some real numbers $c(p)$,
\begin{equation}\label{eq:Atilde}
\widetilde{A}
	=\sum_{\substack{p(u)\in\mathcal{P}_3(u)\\\|p(u)\|=2m+1\\|p(u)|\leq m}}c(p)\int_0^{2\pi}p(u)\d x.
\end{equation}
Let $j\in\li 0,m\ri$. By integration by parts and Leibniz' formula,
\begin{align*}
\overline{\langle D^{m-j}(u D^j\Pi u)|
	D^{m+1}\Pi u\rangle}
	&=\langle D^{m+1-j}(u D^j\Pi u)|D^{m}\Pi u\rangle\\
	&=\langle u D^{m+1}(\Pi u)|D^{m}\Pi u\rangle
	+\langle D^{m+1-j} (u)D^j(\Pi u)|D^{m}\Pi u\rangle\\
	&+\sum_{k=1}^{m-j}\binom{m+1-j}{k}\langle D^{k}(u) D^{m+1-k}(\Pi u)|D^{m}\Pi u\rangle.
\end{align*}
We take the real part and sum over the indices $j$. When distinguishing the cases $j=0$ and $j\geq1$, we see that for some suitable $\widetilde{A}$ as in \eqref{eq:Atilde}, $A$ decomposes as
\[
A=-2(m+1)\Re(\langle u D^{m+1}(\Pi u)|D^{m}\Pi u\rangle)
-2\Re(\langle D^{m+1} (u)\Pi u|D^{m}\Pi u\rangle)+\widetilde{A}.
\]
Write $\Pi u=\frac{u+iHu}{2}$, then there exists some real constants $c(\varepsilon_1,\varepsilon_2)$ such that
\begin{align*}
A
	=\sum_{\varepsilon_1,\varepsilon_2\in\{0,1\}}c(\varepsilon_1,\varepsilon_2)\int_0^{2\pi}u\partial_x^m(H^{\varepsilon_1}u)\partial_x^{m+1}(H^{\varepsilon_2}u)\d x
	-\Re(\langle D^{m+1}(u)iHu|D^m(\Pi u)\rangle)
	+\widetilde{A}.
\end{align*}

On the one hand, the terms in the sum are simplified as follows (see the remark from Tzvetkov and Visciglia \cite{TzvetkovVisciglia2014}). When $\varepsilon_1=\varepsilon_2$,
\begin{align*}
\int_0^{2\pi}u\partial_x^m(H^{\varepsilon_1}u)\partial_x^{m+1}(H^{\varepsilon_1}u)\d x
	&=\half\int_0^{2\pi}u\partial_x((\partial_x^m(H^{\varepsilon_1}u))^2)\d x\\
	&=-\half\int_0^{2\pi}\partial_x(u)(\partial_x^m(H^{\varepsilon_1}u))^2\d x
\end{align*}
so this term is a remainder term to be added to $\widetilde{A}$. Moreover, by integration by parts,
\begin{align*}
\int_0^{2\pi}u\partial_x^m(u)\partial_x^{m+1}(Hu)\d x
	&=-\int_0^{2\pi}\partial_x(u)\partial_x^m(u)\partial_x^{m}(Hu)\d x-\int_0^{2\pi}u\partial_x^{m+1}(u)\partial_x^{m}(Hu)\d x.
\end{align*}
Therefore, the sum can be written as a linear combination of the term $\int_0^{2\pi} u(H\partial_x^mu)\partial_x^{m+1}u\d x$ and other terms that can be added to the remainder $\widetilde{A}$.

On the other end,
\begin{align*}
\Re(\langle D^{m+1}(u)iHu|D^m(\Pi u)\rangle)
	&=\Re(\langle \partial_x^{m+1}(u)Hu|\partial_x^m(\Pi u)\rangle)\\
	&=\half\left(\Re(\langle \partial_x^{m+1}(u)Hu|\partial_x^m(u)\rangle)
	+\Re(\langle \partial_x^{m+1}(u)Hu|i\partial_x^m(H u)\rangle)\right)
\end{align*}
Since $u$ is real valued, so is $Hu$, therefore
\[
\Re(\langle \partial_x^{m+1}(u)Hu|i\partial_x^m(H u)\rangle)=0.\]
By integration by parts, we then write
\begin{align*}
\Re(\langle D^{m+1}(u)iHu|D^m(\Pi u)\rangle)
	&=\frac{1}{4\pi}\int_0^{2\pi}\partial_x^{m+1}(u)Hu\partial_x^m(u)\d x\\
	&=-\frac{1}{8\pi}\int_0^{2\pi}\partial_x(Hu)(\partial_x^m(u))^2\d x
\end{align*}
as a remainder term to be added to $\widetilde{A}$.

$\bullet$ We now tackle term $B$ , for which we have applied $T_u$ at least twice. We show that it can be written for some real numbers $c(p)$ as a sum
\[
B=\sum_{j=4}^{2m+4}\sum_{\substack{p(u)\in\mathcal{P}_j(u)\\\|p(u)\|=2m+4-j\\|p(u)|\leq m}}c(p)\int_0^{2\pi}p(u)\d x.
\]
 Let $\widetilde{B}$ be one of the terms in $B$ obtained by applying $T_u$ $(j-1)$ times on the left side and $(k-1)$ times on the right side.

Assume that we have applied $T_u$ at least once in each side of the brackets, i.\@e.\@ $j-1\in\li 1,m+1\ri$ and $k-1\in\li 1,m+1\ri$. Then we can apply Leibniz' rule and decompose the left side as a complex linear combination of terms of the form $p(u)$ where $p(u)\in\mathcal{P}_j(u)$, $\|p(u)\|=m+2-j$ and $|p(u)|\leq m$ (for the right side we just replace $j$ by $k$). The term $\widetilde{B}$ is therefore a complex linear combination of terms $\int_0^{2\pi}p(u)\d x$, where $p(u)\in\mathcal{P}_l(u)$ for some $l=j+k\in\li 4,2m+4\ri$, $\|p(u)\|=2m+4-l$ and $|p(u)|\leq m$.

Otherwise, we have applied $T_u$ at least twice in the same side of the brackets, let us say the left, and we only have applied the operator $D$ on the other side : $j-1\in\li 2,m+1\ri$ and $k-1=0$. Again by Leibniz' rule, $\widetilde{B}$ decomposes as a sum 
\[
\widetilde{B}=\sum_{j= 3}^{m+2}\sum_{\substack{p(u)\in\mathcal{P}_j(u)\\\|p(u)\|=m+2-j\\|p(u)|\leq m-1}}c(p)\langle p(u)|D^{m+1}\Pi u\rangle.
\]
But then by integration by parts and Leibniz' rule again,
\begin{align*}
\widetilde{B}
	&=\sum_{j= 3}^{m+2}\sum_{\substack{p(u)\in\mathcal{P}_j(u)\\\|p(u)\|=m+2-j\\|p(u)|\leq m-1}}c(p)\overline{\langle Dp(u)|D^{m}\Pi u\rangle}\\
	&=\sum_{j= 3}^{m+2}\sum_{\substack{p(u)\in\mathcal{P}_j(u)\\\|p(u)\|=m+3-j\\|p(u)|\leq m}}c'(p)\overline{\langle p(u)|D^{m}\Pi u\rangle}\\
	&=\sum_{j= 4}^{m+3}\sum_{\substack{p(u)\in\mathcal{P}_j(u)\\\|p(u)\|=2m+4-j\\|p(u)|\leq m}}c''(p)\int_0^{2\pi}p(u)\d x,
\end{align*}
which is of the desired form.
\end{proof}

\bibliography{/home/gassot/Documents/these/references/mybib.bib}{}

\begin{thebibliography}{10}

\bibitem{AmickToland1991}
C.~J. Amick and J.~F. Toland.
\newblock {Uniqueness and related analytic properties for the Benjamin-Ono
  equation—a nonlinear Neumann problem in the plane}.
\newblock {\em Acta Mathematica}, 167(1):107--126, 1991.

\bibitem{Benjamin1967}
T.~B. Benjamin.
\newblock {Internal waves of permanent form in fluids of great depth}.
\newblock {\em Journal of Fluid Mechanics}, 29(3):559–592, 1967.

\bibitem{BockKruskal1979}
T.~Bock and M.~Kruskal.
\newblock {A two-parameter Miura transformation of the Benjamin-Ono equation}.
\newblock {\em Physics Letters A}, 74(3-4):173--176, 1979.

\bibitem{FengHan1996}
X.~Feng and X.~Han.
\newblock {On the Cauchy problem for the third order Benjamin-Ono equation}.
\newblock {\em Journal of the London Mathematical Society}, 53(3):512--528,
  1996.

\bibitem{GerardKappeler2019}
P.~G{\'e}rard and T.~Kappeler.
\newblock {On the integrability of the Benjamin-Ono equation on the torus}.
\newblock {\em arXiv preprint arXiv:1905.01849, to appear in CPAM}, 2019.

\bibitem{GerardKappelerTopalov2019}
P.~G{\'e}rard, T.~Kappeler, and P.~Topalov.
\newblock {On the flow map of the Benjamin-Ono equation on the torus}.
\newblock {\em arXiv preprint arXiv:1909.07314}, 2019.

\bibitem{LinaresPilodPonce2011}
F.~Linares, D.~Pilod, and G.~Ponce.
\newblock {Well-posedness for a higher-order Benjamin--Ono equation}.
\newblock {\em Journal of Differential Equations}, 250(1):450--475, 2011.

\bibitem{Matsuno1984}
Y.~Matsuno.
\newblock {\em {Bilinear transformation method}}.
\newblock Mathematics in Science and Engineering. Elsevier, Burlington, MA,
  1984.

\bibitem{MolinetPilod2012}
L.~Molinet and D.~Pilod.
\newblock {Global well-posedness and limit behavior for a higher-order
  Benjamin-Ono equation}.
\newblock {\em Communications in Partial Differential Equations},
  37(11):2050--2080, 2012.

\bibitem{Nakamura1979-2}
A.~Nakamura.
\newblock {A direct method of calculating periodic wave solutions to nonlinear
  evolution equations. I. Exact two-periodic wave solution}.
\newblock {\em Journal of the Physical Society of Japan}, 47(5):1701--1705,
  1979.

\bibitem{Nakamura1979}
A.~Nakamura.
\newblock {Bäcklund transform and conservation laws of the Benjamin-Ono
  equation}.
\newblock {\em Journal of the Physical Society of Japan}, 47(4):1335--1340,
  1979.

\bibitem{Ono1977}
H.~Ono.
\newblock {Algebraic solitary waves in stratified fluids}.
\newblock {\em Journal of the Physical Society of Japan}, 39(4):1082--1091,
  1975.

\bibitem{PavaHakkaev2010}
J.~A. Pava and S.~Hakkaev.
\newblock {Ill-posedness for periodic nonlinear dispersive equations}.
\newblock {\em Electronic Journal of Differential Equations}, 2010(119):1--19,
  2010.

\bibitem{PavaNatali2008}
J.~A. Pava and F.~M. Natali.
\newblock {Positivity properties of the Fourier transform and the stability of
  periodic travelling-wave solutions}.
\newblock {\em SIAM Journal on Mathematical Analysis}, 40(3):1123--1151, 2008.

\bibitem{Saut2018}
J.-C. Saut.
\newblock {Benjamin-Ono and Intermediate Long Wave equation: modeling, IST and
  PDE}.
\newblock {\em arXiv preprint arXiv:1811.08652, to appear in Nonlinear
  Dispersive Partial Differential Equations and Inverse Scattering, Fields
  Institute Communications 83, P.Miller, P.Perry, J.-C. Saut, C. Sulem eds,
  Springer}, 2018.

\bibitem{Tanaka2019-2}
T.~Tanaka.
\newblock {Local well-posedness for fourth order Benjamin-Ono type equations}.
\newblock {\em arXiv preprint arXiv:1902.06452}, 2019.

\bibitem{Tanaka2019}
T.~Tanaka.
\newblock {Local well-posedness for third order Benjamin-Ono type equations on
  the torus}.
\newblock {\em Advances in Differential Equations}, 24(9/10):555--580, 2019.

\bibitem{TzvetkovVisciglia2014}
N.~Tzvetkov and N.~Visciglia.
\newblock {Invariant Measures and Long-Time Behavior for the Benjamin–Ono
  Equation}.
\newblock {\em International Mathematics Research Notices},
  2014(17):4679--4714, 05 2013.

\end{thebibliography}
\bibliographystyle{abbrv}
\Addresses

\end{document}